\documentclass[12pt]{amsart}
\usepackage{amsfonts,amsmath,amsthm,amssymb,mathrsfs}
\usepackage{MnSymbol}
\usepackage{tikz, tikz-cd}\usetikzlibrary{arrows}
\usetikzlibrary{decorations.markings}
\usepackage[square]{natbib}
\setcitestyle{numbers}

\usepackage[all]{xy}
\usepackage[vcentermath]{youngtab}
\usepackage{ytableau}
\usepackage{hyperref}

\usepackage[margin=1.1in]{geometry}

\definecolor{darkred}{rgb}{1,0,0}

\newtheorem{theorem}{Theorem}[section]
\newtheorem{maintheorem}{Theorem}	

\newtheorem{lemma}[theorem]{Lemma}

\newtheorem{proposition}[theorem]{Proposition}

\theoremstyle{remark}

\theoremstyle{definition}

\newtheorem{remark}[theorem]{Remark}
\newtheorem{example}[theorem]{Example}
\newtheorem{definition}[theorem]{Definition}

\newtheorem*{claim*}{Claim}
\newtheorem*{defn*}{Definition}

\newcommand{\mci}{\mathcal{I}}

\newcommand{\lam}{\lambda}

\newcommand\preceqdot{\mathrel{\ooalign{$\prec$\cr
  \hidewidth\raise0.225ex\hbox{$\cdot\mkern0.5mu$}\cr}}}

\newcommand{\rmci}{\vec{\mci}}

\title{Webs and canonical bases in degree two}
\author{Chris Fraser}
\address{Michigan State University, Lansing, MI.}
\email{fraserc4@msu.edu}
\keywords{Grassmannians, webs, canonical basis, Catalan}

\numberwithin{equation}{section}

\begin{document}
\begin{abstract}We show that Lusztig's canonical basis for the degree two part of the Grassmannian coordinate ring is given by SL(k) web diagrams. Equivalently, we show that every SL(2) web immanant of a plabic graph for Gr(k,n) is an SL(k) web invariant. 
\end{abstract}

\maketitle

Let $L_n(\lambda)$ denote the irreducible polynomial representation of ${\rm GL}_n$ indexed by the weakly decreasing sequence $\lambda \in \mathbb{N}^n$. 
Its character is the Schur polynomial $s_\lambda(x_1,\dots,x_n)$ and its dimension is the cardinality of the set ${\rm SSYT}(\lambda,[n])$ of 
semistandard Young tableaux whose shape is~$\lambda$ and whose entries are drawn from $[n]:= \{1,\dots,n\}$. One is interested in constructing bases of $L_n(\lambda)$ with good symmetry and positivity properties and with basis elements indexed by ${\rm SSYT}(\lambda,[n])$ in a natural way, see \cite{Kamnitzer} for a recent survey. Lusztig and Kashiwara introduced the {\sl dual canonical basis} (henceforth the ``canonical'' basis) as one solution to this problem. An ongoing aim of combinatorial representation theory is to make this basis more explicit. 

Let $\omega_k$ denote the $k$th fundamental weight for ${\rm GL}_n$. In the special case that $\lambda = d \omega_k$ for a positive integer~$d$, one can study the representation $L_n(\lambda)$ and its canonical basis via the technology of 
${\rm SL}_k$ {\sl web diagrams}. These are planar diagrams drawn in an $n$-gon and subject to a ``$k$-valency condition.'' The number of boundary edges of such a diagram is always a multiple of~$k$. Such a web $W$ determines a {\sl web invariant} 
$$[W] \in L_n(d\omega_k)$$ 
when $W$ has $dk$ many boundary edges. Web invariants are known to span the representation~$L_n(d\omega_k)$. They are not linearly independent, but all linear relations between them are a consequence of known diagrammatical relations, the {\sl skein relations}~\cite{CKM}. A second proof of this fact using Postnikov's {\sl plabic graphs} appeared in \cite{FLL}.

The special case of the representations $L_n(d\omega_k)$ is of some interest because of the following connection with the Grassmannian ${\rm Gr}(k,n)$ of $k$-subspaces in $\mathbb{C}^n$. Let $\mathbb{C}[{\rm Gr}(k,n)]$ denote the homogeneous coordinate ring of ${\rm Gr}(k,n)$ in the Pl\"ucker embedding and let 
$\mathbb{C}[{\rm Gr}(k,n)]_{(d)}$ denote the subspace spanned by degree~$d$ monomials in Pl\"ucker coordinates. The action of ${\rm GL}_n$ on $\mathbb{C}^n$ induces an action on the Grassmannian coordinate ring and one has
$$ \mathbb{C}[{\rm Gr}(k,n)]_{(d)} \cong L_n(d\omega_k)$$
as ${\rm GL}_n$-representations. Thus the coordinate ring inherits a canonical basis degree by degree.

In this realization, the canonical basis when $d=1$ is identified with the set of Pl\"ucker coordinates on ${\rm Gr}(k,n)$. Our main result addresses the case $d=2$:
\begin{maintheorem}\label{thm:main1}
Each element of the dual canonical basis for $L_n(2\omega_k)$ is an ${\rm SL}_k$ web invariant, for any $k$ and $n$. 
\end{maintheorem}

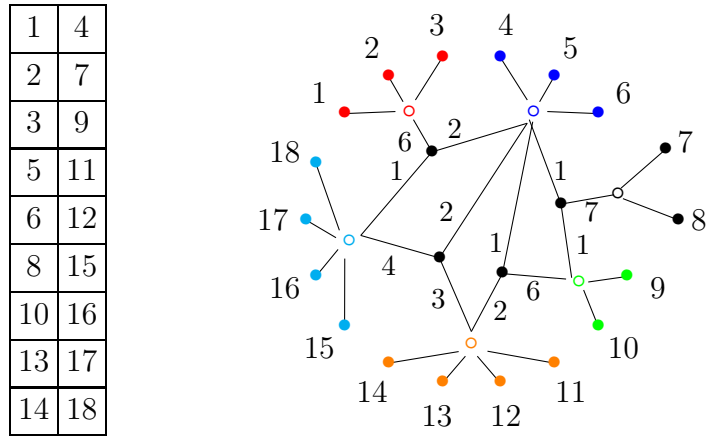
\begin{figure}\centering
\begin{tikzpicture}[scale = 1.1]
\node at (-5,0) {
\begin{ytableau}
1 & 4  \\
2 & 7  \\
3 & 9  \\
5 & 11  \\
6 & 12 \\
8 & 15  \\
10 & 16  \\
13 & 17  \\
14 & 18  \\
\end{ytableau}};

\draw (0:2.5cm)--(.4*360/18:1.85cm);
\draw (1*360/18:2.5cm)--(.6*360/18:1.85cm);
\node at (.5*360/18:1.8cm) {$\circ$};
\node at (0:2.5cm) {$\bullet$};
\node at (360/18:2.5cm) {$\bullet$};

\draw (16*360/18:2cm)--(16.4*360/18:1.6cm);
\draw (17*360/18:2cm)--(16.6*360/18:1.6cm);
\node at (16.5*360/18:1.5cm) {\textcolor{green}{$\circ$}};
\node at (16*360/18:2cm) {\textcolor{green}{$\bullet$}};
\node at (17*360/18:2cm) {\textcolor{green}{$\bullet$}};

\draw (2*360/18:2cm)--(-5+3*360/18:1.6cm);
\draw (4*360/18:2cm)--(5+3*360/18:1.6cm);
\draw (3*360/18:2cm)--(3*360/18:1.65cm);
\node at (3*360/18:1.5cm) {\textcolor{blue}{$\circ$}};
\node at (2*360/18:2cm) {\textcolor{blue}{$\bullet$}};
\node at (3*360/18:2cm) {\textcolor{blue}{$\bullet$}};
\node at (4*360/18:2cm) {\textcolor{blue}{$\bullet$}};

\draw (5*360/18:2cm)--(-5+6*360/18:1.6cm);
\draw (7*360/18:2cm)--(5+6*360/18:1.6cm);
\draw (6*360/18:2cm)--(6*360/18:1.65cm);
\node at (6*360/18:1.5cm) {\textcolor{red}{$\circ$}};
\node at (5*360/18:2cm) {\textcolor{red}{$\bullet$}};
\node at (6*360/18:2cm) {\textcolor{red}{$\bullet$}};
\node at (7*360/18:2cm) {\textcolor{red}{$\bullet$}};

\draw (8*360/18:2cm)--(-7.5+9.6*360/18:1.6cm);
\draw (9*360/18:2cm)--(-2.5+9.5*360/18:1.65cm);
\draw (10*360/18:2cm)--(2.5+9.5*360/18:1.65cm);
\draw (11*360/18:2cm)--(5+9.6*360/18:1.6cm);
\node at (9.5*360/18:1.5cm) {\textcolor{cyan}{$\circ$}};
\node at (8*360/18:2cm) {\textcolor{cyan}{$\bullet$}};
\node at (9*360/18:2cm) {\textcolor{cyan}{$\bullet$}};
\node at (10*360/18:2cm) {\textcolor{cyan}{$\bullet$}};
\node at (11*360/18:2cm) {\textcolor{cyan}{$\bullet$}};

\draw (12*360/18:2cm)--(-7.5+13.6*360/18:1.6cm);
\draw (13*360/18:2cm)--(-2.5+13.5*360/18:1.65cm);
\draw (14*360/18:2cm)--(2.5+13.5*360/18:1.65cm);
\draw (15*360/18:2cm)--(5+13.6*360/18:1.6cm);
\node at (13.5*360/18:1.5cm) {\textcolor{orange}{$\circ$}};
\node at (12*360/18:2cm) {\textcolor{orange}{$\bullet$}};
\node at (13*360/18:2cm) {\textcolor{orange}{$\bullet$}};
\node at (14*360/18:2cm) {\textcolor{orange}{$\bullet$}};
\node at (15*360/18:2cm) {\textcolor{orange}{$\bullet$}};

\node at (-.8,.95) {\small 6};
\node at (-.2,1.1) {\small 2};
\node at (-.3,.1) {\small 2};
\node at (.3,-.28) {\small 1};
\node at (1.08,.6) {\small 1};
\node at (1.35,-.3) {\small 1};
\node at (-1.0,-.55) {\small 4};
\node at (-.9,.60) {\small 1};
\node at (-.4,-.95) {\small 3};
\node at (.35,-1.1) {\small 2};
\node at (.75,-.85) {\small 6};
\node at (1.45,.1) {\small 7};

\node at (6*360/18:.95cm) {$\bullet$};
\draw (6*360/18:1.4cm)--(6*360/18:.95cm);
\draw (-2+9.5*360/18:1.35cm)--(6*360/18:.95cm);
\draw (3*360/18:1.35cm)--(6*360/18:.95cm);

\node at (11.5*360/18:.6cm) {$\bullet$};
\draw (-2+9.5*360/18:1.35cm)--(11.5*360/18:.6cm);
\draw (3*360/18:1.35cm)--(11.5*360/18:.6cm);
\draw (13.5*360/18:1.35cm)--(11.5*360/18:.6cm);

\node at (15*360/18:.75cm) {$\bullet$};
\draw (-2+3*360/18:1.4cm)--(15*360/18:.75cm);
\draw (13.5*360/18:1.35cm)--(15*360/18:.75cm);
\draw (-2+16.5*360/18:1.35cm)--(15*360/18:.75cm);

\node at (.5*360/18:1.1cm) {$\bullet$};
\draw (3*360/18:1.4cm)--(.5*360/18:1.1cm);
\draw (.5*360/18:1.75cm)--(.5*360/18:1.1cm);
\draw (16.5*360/18:1.4cm)--(.5*360/18:1.1cm);

\node at (360/18:2.75cm) {7};
\node at (2*360/18:2.4cm) {6};
\node at (3*360/18:2.4cm) {5};
\node at (4*360/18:2.4cm) {4};
\node at (5*360/18:2.4cm) {3};
\node at (6*360/18:2.4cm) {2};
\node at (7*360/18:2.4cm) {1};
\node at (8*360/18:2.4cm) {18};
\node at (9*360/18:2.4cm) {17};
\node at (10*360/18:2.4cm) {16};
\node at (11*360/18:2.4cm) {15};
\node at (12*360/18:2.4cm) {14};
\node at (13*360/18:2.4cm) {13};
\node at (14*360/18:2.4cm) {12};
\node at (15*360/18:2.4cm) {11};
\node at (16*360/18:2.4cm) {10};
\node at (17*360/18:2.4cm) {9};
\node at (0:2.75cm) {8};
\end{tikzpicture}
\caption{A rectangular semistandard Young tableau with 9 rows and 2 columns and its corresponding ${\rm SL}_9$ web diagram on 18 boundary vertices. The main construction of this paper formulates such a recipe for rectangular tableaux with $k$ rows and $2$ columns for any~$k$.}
\label{fig:mainexample}
\end{figure}

We prove our theorem constructively: we set up a multivalued {\sl tableau-to-web map} 
$${\rm SSYT}(2 \omega_k,[n]) \to \text{ ${\rm SL}_k$ web diagrams on $n$ vertices with $2k$ boundary edges}$$
which becomes well-defined once we pass from web diagrams to web invariants. See Figure~\ref{fig:mainexample} for an instance of this map in the case $k=9$ which serves as our running example. 

Our tableau-to-web map passes through several intermediate Catalan-style objects, see Figure~\ref{fig:schematic} for a schematic of these intermediate steps and Figure~\ref{fig:intermediateobjects} for an illustration of these intermediate objects. To produce a web diagram from a tableau using our recipe, one first constructs a dissection of a certain polygon and then extends this dissection to a triangulation. The latter step involves choices but we give two different proofs that the resulting web invariant is independent of such choice: a direct proof which interprets the flip move on triangulations as a skein relation between web diagrams (see Proposition~\ref{prop:flipinvariance}) and a conceptual proof which shows that our web diagrams satisfy a duality property which specifies their web invariants uniquely (see Theorem~\ref{thm:duality}).

The {\sl cluster algebra} structure on $\mathbb{C}[{\rm Gr}(k,n)]$ suggests that 
the notions of web invariant and canonical basis element should be intimately related: every {\sl cluster monomial} in $\mathbb{C}[{\rm Gr}(k,n)]$ is expected to be both a web invariant \cite{FPI} and a canonical basis element (see \cite{GLS,Qin} for closely related results). One knows, however, that there are canonical basis elements which are not web invariants and also that there are canonical basis elements which 
are not cluster monomials (the latter happens even in degree $d=2$).

Expressing a canonical basis element $b$ as a web invariant is a useful thing to do. Due to the Pl\"ucker relations, there are manifold ways of expressing $b$ as a polynomial in Pl\"ucker coordinates. One can remove this ambiguity by expressing $b$ as a $\mathbb{Z}$-linear combination of {\sl standard} monomials in Pl\"ucker coordinates, but this expression does not respect the cyclic symmetry of the $n$-gon and can be a bit tedious to work out. By contrast, one can ``see'' many properties of $b$ (e.g., its degree of cyclic symmetry, the {\sl descents} of its corresponding tableau, and so on) directly from a web diagram whose invariant is~$b$.

\begin{figure}[ht]\centering
\begin{tikzpicture}[scale = .9]
\node at (0,0) {SYTs};
\draw [-to] (0,-.3)--(0,-1.0);
\draw [-to] (6,-1.0)--(6,-.25);
\draw [dashed] (1.5,-1.25)--(4,-1.25);
\node at (4,-1.25) {\small $>$};
\draw [dashed] (1.5,0)--(4,0);
\node at (4,0) {\small $>$};
\node at (6,0) {web diagrams};
\draw (7.5,0)--(9.5,0);
\node at (9.3,0) {\small $>$};
\node at (9.5,0) {\small $>$};
\node at (11.5,0) {web invariants};
\node at (0,-1.25) {dissections};
\node at (6,-1.25) {triangulations};
\end{tikzpicture}
\caption{The tableau-to-web map sends a standard tableau~$T$ to a weighted dissection $\mathfrak{d}_T$ of a certain polygon. One makes a choice of triangulation $\mathfrak{t}_T$ extending $\mathfrak{d}_T$ which determines a web diagram $W(\mathfrak{t}_T)$. The vertical maps are reversible but the dashed horizontal maps are multivalued. Nonetheless the invariant $[W(\mathfrak{t}_T)]$ is well-defined.} \label{fig:schematic}
\end{figure}
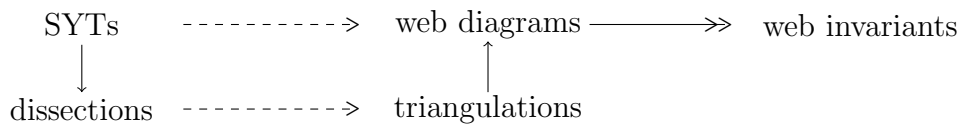

{\bf Previous work.} There has been much previous work comparing canonical bases and web invariants \cite{FPI,FrenkelKhovanov,KhovanovKuperberg,Kuperberg,Lamberti}. For fixed $d,k$ the complexity of the underlying combinatorics stabilizes once $n \geq dk$. We can analyze the problem by fixing either $k$ or $d$, letting the other parameter vary and assuming that $n$ is in the stable range.

When $k=2$, every canonical basis element is an ${\rm SL}_2$ web invariant. The underlying combinatorics is that of noncrossing matchings and the Temperley-Lieb algebra. Once $k>2$, web invariants are linearly dependent and it becomes an interesting problem to find a basis of web invariants with good properties. This problem has a nice solution when $k=3$ \cite{Kuperberg} but remains elusive when $k>3$. Our main theorem solves this problem for all~$k$ in the special case $d=2$.

The degree $d$ part of the Grassmannian coordinate ring when $d$ is either 2 or 3 has been actively studied, mostly through the lens of deciding which canonical basis elements are cluster variables \cite{BBE,BBEL,BBL,LeYildirim}. Our approach complements these papers by studying the whole basis (when $d=2$), not only the cluster variables.

{\bf Cyclic symmetry.} The canonical basis of $L_n(d\omega_k)$ was used to establish the cyclic sieving phenomenon for the action of promotion on ${\rm SSYT}(d\omega_k,[n])$ \cite{Rhoades}. The crucial property is that the canonical basis is a weight basis whose elements are permuted by the automorphism of $\mathbb{C}[{\rm Gr}(k,n)]$ induced by rotation of the $n$-gon. This action on basis elements implements promotion on tableaux. It remains an interesting open problem to find a basis of web invariants with the same properties. Again, this problem is solved when $k=2,3$ and now holds when $d=2$ using our Theorem~\ref{thm:main1}. We hope our construction might be a useful input towards solving the problem when $d=3$, and perhaps for all~$d$.

{\bf Web immanants.} The well-studied relationship between ${\rm Gr}(k,n)$ and {\sl plabic graphs}~\cite{Postnikov} yields a different perspective on the canonical basis for $L_n(2\omega_k)$. In this viewpoint, the canonical basis for $L_n(2\omega_k)$ coincides with the set of ${\rm SL}_2$ {\sl web immanants} \cite{LamDemazure,LamDimers}. 
The web immanant is constructed somewhat indirectly using the double dimer model 
on a plabic network $N$. Roughly speaking, web immanants are indexed by noncrossing matchings $m$, and the web immanant $F_m$ is defined as the generating function for double dimer covers of $N$ which match boundary vertices according to~$m$. It is far from obvious that the ``${\rm SL}_2$-style'' combinatorics of noncrossing matchings should have any relationship with the combinatorics of ${\rm SL}_k$ web diagrams. Nonetheless, since ${\rm SL}_2$ web immanants and degree-two canonical basis elements coincide, Theorem~\ref{thm:main1} can be reformulated as follows:
\begin{maintheorem}\label{thm:main2}
Every SL$_2$ web immanant for {\rm Gr}(k,n) is an SL$_k$ web invariant, for any $k,n$.
\end{maintheorem}
This is how we will approach Theorem~\ref{thm:main1}. Theorem~\ref{thm:main2} was anticipated in \cite[Observation 8.2]{FLL}, where we observed that it held when $k=2,3,4,5$. 

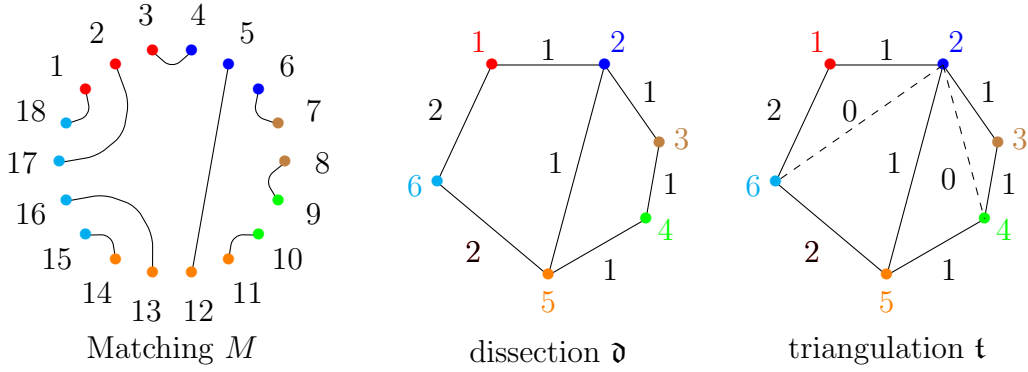
\begin{figure}\centering
\begin{tikzpicture}
\draw [rounded corners] (360/18:1.5cm)--(1.5*360/18:1.25cm)--(2*360/18:1.5cm);
\draw [rounded corners] (4*360/18:1.5cm)--(4.5*360/18:1.25cm)--(5*360/18:1.5cm);
\draw [rounded corners] (7*360/18:1.5cm)--(7.5*360/18:1.25cm)--(8*360/18:1.5cm);
\draw [rounded corners] (11*360/18:1.5cm)--(11.5*360/18:1.25cm)--(12*360/18:1.5cm);
\draw [rounded corners] (15*360/18:1.5cm)--(15.5*360/18:1.25cm)--(16*360/18:1.5cm);
\draw [rounded corners] (17*360/18:1.5cm)--(17.5*360/18:1.25cm)--(18*360/18:1.5cm);
\draw [rounded corners] (6*360/18:1.5cm)--(6.25*360/18:1.0cm)--(7.5*360/18:.75cm)--(8.75*360/18:1.0cm)--(9*360/18:1.5cm);
\draw [rounded corners] (10*360/18:1.5cm)--(10.25*360/18:1.0cm)--(11.5*360/18:.75cm)--(12.75*360/18:1.0cm)--(13*360/18:1.5cm);
\draw (3*360/18:1.5cm)--(14*360/18:1.5cm);
\node at (0:1.5cm) {\textcolor{brown}{$\bullet$}};
\node at (360/18:1.5cm) {\textcolor{brown}{$\bullet$}};
\node at (2*360/18:1.5cm) {\textcolor{blue}{$\bullet$}};
\node at (3*360/18:1.5cm) {\textcolor{blue}{$\bullet$}};
\node at (4*360/18:1.5cm) {\textcolor{blue}{$\bullet$}};
\node at (5*360/18:1.5cm) {\textcolor{red}{$\bullet$}};
\node at (6*360/18:1.5cm) {\textcolor{red}{$\bullet$}};
\node at (7*360/18:1.5cm) {\textcolor{red}{$\bullet$}};
\node at (8*360/18:1.5cm) {\textcolor{cyan}{$\bullet$}};
\node at (9*360/18:1.5cm) {\textcolor{cyan}{$\bullet$}};
\node at (10*360/18:1.5cm) {\textcolor{cyan}{$\bullet$}};
\node at (11*360/18:1.5cm) {\textcolor{cyan}{$\bullet$}};
\node at (12*360/18:1.5cm) {\textcolor{orange}{$\bullet$}};
\node at (13*360/18:1.5cm) {\textcolor{orange}{$\bullet$}};
\node at (14*360/18:1.5cm) {\textcolor{orange}{$\bullet$}};
\node at (15*360/18:1.5cm) {\textcolor{orange}{$\bullet$}};
\node at (16*360/18:1.5cm) {\textcolor{green}{$\bullet$}};
\node at (17*360/18:1.5cm) {\textcolor{green}{$\bullet$}};

\node at (360/18:2cm) {7};
\node at (2*360/18:2cm) {6};
\node at (3*360/18:2cm) {5};
\node at (4*360/18:2cm) {4};
\node at (5*360/18:2.0cm) {3};
\node at (6*360/18:2.0cm) {2};
\node at (7*360/18:2.0cm) {1};
\node at (8*360/18:2cm) {18};
\node at (9*360/18:2cm) {17};
\node at (10*360/18:2cm) {16};
\node at (11*360/18:2cm) {15};
\node at (12*360/18:2cm) {14};
\node at (13*360/18:2cm) {13};
\node at (14*360/18:2cm) {12};
\node at (15*360/18:2cm) {11};
\node at (16*360/18:2cm) {10};
\node at (17*360/18:2cm) {9};
\node at (0:2cm) {8};
\node at (0,-2.5) {Matching $M$};

\begin{scope}[xshift = 5cm]
\draw (3*360/18:1.5cm)--(6*360/18:1.5cm)--(9.5*360/18:1.5cm)--(13.5*360/18:1.5cm)--(16.5*360/18:1.5cm)--(.5*360/18:1.5cm)--(3*360/18:1.5cm);
\draw (3*360/18:1.5cm)--(13.5*360/18:1.5cm);
\node at (.5*360/18:1.5cm) {\textcolor{brown}{$\bullet$}};
\node at (.5*360/18:1.8cm) {\textcolor{brown}{3}};
\node at (3*360/18:1.5cm) {\textcolor{blue}{$\bullet$}};
\node at (3*360/18:1.85cm) {\textcolor{blue}{2}};
\node at (6*360/18:1.85cm) {\textcolor{red}{1}};
\node at (9.5*360/18:1.8cm) {\textcolor{cyan}{6}};
\node at (13.5*360/18:1.85cm) {\textcolor{orange}{5}};
\node at (16.5*360/18:1.8cm) {\textcolor{green}{4}};
\node at (6*360/18:1.5cm) {\textcolor{red}{$\bullet$}};
\node at (9.5*360/18:1.5cm) {\textcolor{cyan}{$\bullet$}};
\node at (11.5*360/18:1.55cm) {\textcolor{red}{2}};
\node at (13.5*360/18:1.5cm) {\textcolor{orange}{$\bullet$}};
\node at (16.5*360/18:1.5cm) {\textcolor{green}{$\bullet$}};
\node at (7.75*360/18:1.65cm) {2};
\node at (11.5*360/18:1.55cm) {2};
\node at (0,1.5) {1};
\node at (1.75*360/18:1.65cm) {1};
\node at (17.5*360/18:1.65cm) {1};
\node at (15.0*360/18:1.65cm) {1};
\node at (.1,0) {1};
\node at (0,-2.5) {dissection $\mathfrak{d}$};
\end{scope}

\begin{scope}[xshift = 9.5cm]
\draw (3*360/18:1.5cm)--(6*360/18:1.5cm)--(9.5*360/18:1.5cm)--(13.5*360/18:1.5cm)--(16.5*360/18:1.5cm)--(.5*360/18:1.5cm)--(3*360/18:1.5cm);
\draw (3*360/18:1.5cm)--(13.5*360/18:1.5cm);
\node at (.5*360/18:1.5cm) {\textcolor{brown}{$\bullet$}};
\node at (.5*360/18:1.8cm) {\textcolor{brown}{3}};
\node at (3*360/18:1.5cm) {\textcolor{blue}{$\bullet$}};
\node at (3*360/18:1.85cm) {\textcolor{blue}{2}};
\node at (6*360/18:1.85cm) {\textcolor{red}{1}};
\node at (9.5*360/18:1.8cm) {\textcolor{cyan}{6}};
\node at (13.5*360/18:1.85cm) {\textcolor{orange}{5}};
\node at (16.5*360/18:1.8cm) {\textcolor{green}{4}};
\node at (6*360/18:1.5cm) {\textcolor{red}{$\bullet$}};
\node at (9.5*360/18:1.5cm) {\textcolor{cyan}{$\bullet$}};
\node at (11.5*360/18:1.55cm) {\textcolor{red}{2}};
\node at (13.5*360/18:1.5cm) {\textcolor{orange}{$\bullet$}};
\node at (16.5*360/18:1.5cm) {\textcolor{green}{$\bullet$}};
\node at (7.75*360/18:1.65cm) {2};
\node at (11.5*360/18:1.55cm) {2};
\node at (0,1.5) {1};
\node at (1.75*360/18:1.65cm) {1};
\node at (17.5*360/18:1.65cm) {1};
\node at (15.0*360/18:1.65cm) {1};
\node at (.1,0) {1};
\draw [dashed] (3*360/18:1.5cm)--(9.5*360/18:1.5cm);
\node at (6.25*360/18:.85cm) {0};
\draw [dashed] (3*360/18:1.5cm)--(16.5*360/18:1.5cm);
\node at (-.75*360/18:.85cm) {0};
\node at (0,-2.5) {triangulation $\mathfrak{t}$};
\end{scope}
\end{tikzpicture}
\caption{Intermediate objects appearing in the tableau-to-web map. A tableaux $T$ affords a noncrossing matching~$M$ by a standard Catalan bijection. The boundary vertices of such an $M$ are broken into {\sl color sets} incdicated in the figure. Gluing vertices of the same color and keeping track of the number of edges between each color set, the matching $M$ begets a weighted dissection~$\mathfrak{d}$. The triangulation $\mathfrak{t}$ extends this dissection $\mathfrak{d}$ by adding two weight zero diagonals (drawn dashed). By a simple recipe, $\mathfrak{t}$ begets a web diagram, in this case the web from Figure~\ref{fig:mainexample}.}\label{fig:intermediateobjects}
\end{figure}

One also has a good notion of ${\rm SL}_3$ web immanants and these form a rotation-invariant basis for the degree three part of the Grassmannian coordinate ring. It would be very interesting to find an analogue of Theorem~\ref{thm:main2} in the $d=3$ case, see
{\sl loc. cit.} for a discussion when $k=2,3$.

{\bf Duality of symmetric group representations.} Canonical basis elements are indexed by semistandard tableaux. We prove our main result by first considering the weight subspace 
$$L_{dk}(d \omega_k)^{\rm std} \subset L_{n}(d \omega_k)$$
spanned by the canonical basis elements whose tableau is standard. The symmetric group~$\mathfrak{S}_n$ for $n=dk$ acts on $L_n(d\omega_k)$. The subspace $L_{dk}(d \omega_k)^{\rm std}$ appears as an $\mathfrak{S}_n$-irrep. Invoking Schur's lemma, one obtains a uniquely defined (up to homothety) perfect pairing 
\begin{equation*}
\langle, \rangle \colon L_{dk}(d \omega_k)^{\rm std} \otimes  L_{dk}(k \omega_d)^{\rm std} \to \varepsilon
\end{equation*}
where $\varepsilon$ is the sign representation. 

Thus, a choice of basis for $L_{dk}(k \omega_d)^{\rm std}$ induces a dual basis for $L_{dk}(d \omega_k)^{\rm std}$ with respect to $\langle,\rangle$. It seems likely that this notion of duality respects canonical basis vectors up to a predictable sign. We verified this expectation in \cite{FLL} when $d=2$ using ${\rm SL}_2$ web immanants. Thus, to prove Theorem~\ref{thm:main1} on the ``standard part'' of the representation $L_{2k}(2 \omega_k)$, we merely need to identify a set of web diagrams whose corresponding invariants are dual to the canonical ``noncrossing matching'' basis for $L_{2k}(k \omega_2)^{\rm std}$. This result is the heart of our paper, formulated as Theorem~\ref{thm:duality}. For instance, the ${\rm SL}_9$ web in Figure~\ref{fig:mainexample} is dual to the ${\rm SL}_2$ web given by the noncrossing matching $M$ in Figure~\ref{fig:intermediateobjects}. 

We extend the construction from $L_{2k}(2 \omega_k)^{\rm std}$ to the whole representation $L_{n}(2 \omega_k)$ using routine ``semistandardization'' constructions.

{\bf Organization.} The new results and constructions are in the first three sections. Section~\ref{secn:construction} formulates the map from standard tableaux to web diagrams in purely combinatorial terms and proves the invariance under change of triangulation using skein relations. Section~\ref{secn:positroids} makes a connection between the tableau-to-web map and plabic graphs. Each ${\rm SL}_k$ web diagram is in particular a plabic graph, hence encodes a positroid, and we identify the positroids which arise via the resulting tableau-to-positroid map. Section~\ref{secn:duality} states and proves our core combinatorial result, the web-matching duality Theorem~\ref{thm:duality}. Sections \ref{secn:invariants} and \ref{secn:immanants} give quick definitions of {\sl web invariants} and {\sl web immanants} respectively. Section \ref{secn:sszn} extends all constructions from standard to semistandard tableaux. We then deduce Theorem~\ref{thm:main2}, hence also Theorem~\ref{thm:main1}, from Theorem~\ref{thm:duality} by a quick argument.

We use standard combinatorial notation $\binom{[n]}k$ for the collection of $k$-subsets of $[n]$. We assume familiarity with the notions of standard and semistandard Young tableaux. 

\section{Tableau-to-web map: standard case}\label{secn:construction}
We formulate the combinatorics of our ``degree two'' tableau-to-web map and provide a running example. We follow the conventions in \cite{FLL} and refer to that paper for a more leisurely tour through the definitions.

\begin{definition}\label{defn:webdiagram}[Web diagram]
An ${\rm SL}_k$ {\sl web diagram} $W$ is a planar bipartite graph\\ $W = (V(W),E(W))$ embedded in a disk.  The vertex set $V(W)$ has $n$ {\sl boundary vertices}, each of which is on the boundary of the disk and is colored black in the bipartition. The boundary vertices are numbered $1,\dots,n$ in cyclic order. In addition, $V(W)$ has {\sl interior} vertices drawn in the interior of the disk and colored either white or black. Finally, $W$ comes with a {\sl multiplicity function} ${\rm mult} \colon E(W) \to [0,k]$ with the property that the multiplicity sum around every interior vertex $v$ equals $k$:
$$\sum_{e \sim v}{\rm mult}(e) = k .$$
We consider such $W$ up to isotopy of the disk fixing the boundary vertices. 
\end{definition}

Let us emphasize some of our web conventions which are not always taken in the literature. First, we do not require that every boundary vertex has degree at most one, or that every edge incident to a boundary vertex has multiplicity one. Second, we do not require that internal vertices are trivalent. And third, we get rid of the formalism of tags which is used in \cite{CKM} to fix certain signs, relying instead on results from \cite{FLL} and an extra result on signs (see Lemma~\ref{lem:signs}).

The {\sl content} of an ${\rm SL}_k$ web on $n$ boundary vertices is the degree sequence $(d_1,\dots,d_n)$ of its boundary vertices (with $d_i$ the degree of the $i$th boundary vertex). By elementary arguments, one knows always that $\sum_{i\in [n]}d_i = dk$ for some $d \in \mathbb{N}$ called the {\sl degree} of $W$.

We call a web {\sl type 1} if all of its boundary edges have multiplicity one (i.e., its multiweight is the all-ones vector in $\mathbb{Z}^n$). We call a web {\sl type 2} if all of its boundary vertices have degree at most one. Type 1 webs naturally give rise to functions on Grassmannians, while Type 2 webs are natural from the viewpoint of matchings and immanants.

Note that type 1 webs may have boundary vertices of high degree (although in this paper the degree will never exceed two), and type 2 webs may have boundary edges of high multiplicity (although in this paper the multiplicity will not exceed two). Only these two subclasses of webs will appear in this paper. 

For a type 2 web $W$, the {\sl multiweight} is the sequence $(\nu_1,\dots,\nu_n)$ recording the value of the $W$'s multiplicity function on the $i$th boundary edge of $W$ (i.e., the edge which is connected to boundary vertex $i$ if such exists). If boundary vertex $i$ has degree $i$ we take $\nu_i=0$. 

\begin{remark}
Both the content and the multiweight sequences play a role in what follows, particular when we begin to consider semistandard rather than just standard tableaux. As we spell out below, web duality pairs degree two ${\rm SL}_k$ webs of type 1 with ${\rm SL}_2$ webs. Under this duality, the content of the ${\rm SL}_k$ web is the multiweight of the corresponding ${\rm SL}_2$ web in a certain sense. More precisely, in the language of our Main Theorem~\ref{thm:main2}, one knows that each such ${\rm SL}_k$ web is a web immanant enumerating 2-like subgraphs, and the content of the ${\rm SL}_k$ web corresponds to the multiweight of these 2-like subgraphs.
\end{remark}

We call a web {\sl standard} if it has content $(1,1,\dots,1)$ and all its boundary edges have multiplicity one. Such a web is both type 1 and type 2. Any type 1 web has an associated standard web obtained by replacing each boundary vertex of degree $m$ by $m$ many boundary vertices of degree one (deleting boundary vertices of degree zero). Most arguments involving web diagrams can first be carried out for standard webs and then transferred to arbitrary webs by reversing this process.

\begin{example}
The right diagram in Figure~\ref{fig:mainexample} is a degree two standard ${\rm SL}_9$ web. The black numbers on the outside of the diagram label are the boundary vertex labels and those on the inside indicate edge multiplicities. Edge multiplicities of boundary edges equal one and are omitted. 
\end{example}

By a known recipe which we defer to Definition~\ref{defn:bracketofW}, any type 1 ${\rm SL}_k$ web $W$ of degree $d$ with $n$ many boundary vertices determines an element 
\begin{equation}\label{eq:invariantofW}
[W] \in L_n(d\omega_k) \cong \mathbb{C}[{\rm Gr}(k,n)]_{(d)}.
\end{equation}
To understand the constructions in the current section, the reader need only know that there are certain diagrammatic moves on web diagrams which induce linear relations between web invariants \cite{CKM,FLL}. In our set of conventions, these moves are contraction of bivalent vertices, bigon removal, dipole removal, and the square switch move, see 
\cite[Section 6]{FLL} for a discussion. 

\begin{example}\label{eg:clawgraph}
The simplest example of an ${\rm SL}_k$ web is a claw graph with one interior white vertex joined to $k$ distinct boundary vertices each by an edge of multiplicity one. Letting $I$ denote the subset of boundary vertices which are used in this claw graph, 
the corresponding web invariant is the Pl\"ucker coordinate $\Delta_I \in \mathbb{C}[{\rm Gr}(k,n)]$. 
\end{example}

We associate a web diagram to a tableau via three steps, indicated schematically in Figure~\ref{fig:schematic}. We now describe each of these these steps while reviewing the intermediate combinatorial gadgets which appear along the way.

\subsection{From SYT's to dissections}
We construct a map $T \mapsto \mathfrak{d}_T$ from the set ${\rm SYT}(2\omega_k)$ of standard Young tableaux with $k$ rows and $2$ columns to the set of weighted dissections of polygons of weight~$k$, i.e. 
we formulate the first downwards arrow in Figure~\ref{fig:schematic}. 

We assume familiarity with the concept of noncrossing matching of a $2k$-gon. Such matchings are equinumerous with ${\rm SYT}(2\omega_k)$ and these two sets are related by a well-known {\sl Catalan bijection} described as follows. If $ij$ is an an arc in a noncrossing matching $M$ and if $i <j$ then we say $i$ is the {\sl left endpoint} of $ij$ and $j$ is the {\sl right endpoint}. Given $T \in {\rm SYT}(2\omega_k)$, there exists a unique noncrossing matching $M(T)$ of the $2k$-gon with the property that the left endpoints of $M(T)$ (resp. right endpoints) are exactly the entries of the first (resp. second) column of~$T$. 

\begin{example}
The noncrossing matching $M$ of the 18-gon appearing in Figure~\ref{fig:intermediateobjects} is related to the tableau $T \in {\rm SYT}(2\omega_9)$ from Figure~\ref{fig:mainexample} by the Catalan bijection.  
\end{example}

\begin{definition}[Short arcs]\label{defn:shortarcs}
A {\sl short arc} in a noncrossing matching $M$ is an arc joining adjacent boundary vertices of the $2k$-gon. We let $s(M)$ denote the number of short arcs in~$M$. \end{definition}
Observe that every noncrossing matching has at least two short arcs (provided $k \geq 2$). 

For a noncrossing matching $M$ we can always define uniquely numbers $i_1 < \cdots < i_s$ with the property that $i_j$ and $i_j+1 \mod 2k$ are joined by an arc in $M$ and these are a complete list of the short arcs in $M$. (These typically coincide with the set of left endpoints of the short arcs of $M$ except in the special case that $1$ and $2k$ are joined by a short arc in~$M$.)

\begin{definition}[Color sets of $M$]\label{defn:intervals}
With $i_1,\dots,i_s$ as just defined, the {\sl color sets} of a noncrossing matching $M$ are the cyclic intervals $C_j := (i_{j-1},i_{j}]$.
\end{definition}

The color sets solve the following coloring problem: color, with as few colors as possible, the boundary vertices of the $2k$-gon so that no arc joins two vertices of the same color and all of the color sets are cyclic intervals. 

Observe that we have indexed things in such a way that $1 \in C_1$ always, since $i_s < 1 \leq i_1$ in cyclic order. 

\begin{example}
The matching $M$ from Figure~\ref{fig:intermediateobjects} has six short arcs with left endpoints $1$, $3$, $6$, $8$, $10$, and $14$. On the other hand, the numbers $i_1,\dots,i_6$ as in Definition~\ref{defn:intervals} are 3, 6, 8, 10, 14, and 18. The color sets $C_1, \dots, C_6$ are the intervals $[1,3]$, $[4,6]$, $[7,8]$, $[9,10]$, $[11,14]$, and $[15,18]$, indicated by the colors in the figure.  
\end{example}

\begin{remark}
Directly from the tableau, one can see that entries $i$ and $i+1$ will be in the same color set of $M(T)$ whenever $i$ appears in a strictly higher row of $T$ than $i+1$ does. When this happens, $i$ is often called a {\sl descent} of~$T$. One also needs to decide under what circumstances the pair $\{1,2k\}$ form a descent. Let $T \setminus \{1,2k\}$ denote the tableau obtained by deleting entries 1 and 2k and sliding the entries in the first column up one box. Then $\{1,2k\}$ form a descent if and only if $T \setminus \{1,2k\}$ is not a standard tableau on the ground set $[2,2k-1]$, i.e. if and only if $T \setminus \{1,2k\}$ has a decreasing row. 
\end{remark}

\begin{definition}[Weighted dissections]
A {\sl dissection} of an $s$-gon is a choice of diagonals in the $s$-gon which are mutually noncrossing. (For our purposes, the sides of the $s$-gon are also considered diagonals.) A {\sl weighted dissection} $\mathfrak{d}$ is a dissection together with an $\mathbb{N}$-weighting of its diagonals. The {\sl weight} of $\mathfrak{d}$ is the sum of the weights of its various diagonals. The {\sl content} 
$d_i$ of the vertex~$i$ in $\mathfrak{d}$ is the sum of the weights of its incident diagonals.  
\end{definition}

One has always 
\begin{equation}\label{eq:dissectiondegreesum}
\sum_{i \in [s]}d_i = 2{\rm wt}(\mathfrak{d})
\end{equation}
since each diagonal is counted twice in the above sum.  

A {\sl triangulation} is a dissection which is maximal by inclusion. 


{\bf Construction.} Given a noncrossing matching $M = M(T)$ with $s$ many short arcs and with color sets $C_1,\dots,C_s$, we obtain a weighted dissection $\mathfrak{d}_T$ of the $s$-gon by merging the boundary vertices in each color set (merging color set $C_i$ into vertex $i$ of the $s$-gon), replacing multiple edges by a single edge of the same weight. 

\medskip
This construction is illustrated in the passage from $M$ to $\mathfrak{d}$ in Figure~\ref{fig:intermediateobjects}.

Observe that the total weight of $\mathfrak{d}_T$ is the number of arcs in $M$, which is $k$. Similarly, the content $d_i$ of $\mathfrak{d}_T$ equals the cardinality of the color set $C_i$ of $M$.

\begin{remark}\label{rmk:reversibledissection}
The mapping $M(T) \mapsto \mathfrak{d}_T$ is not reversible. For example, rotating the matching~$M$ in Figure~\ref{fig:intermediateobjects} counterclockwise by one or two units would not affect the corresponding $\mathfrak{d}_T$.

More generally, the size of the inverse image of a given weighted dissection $\mathfrak{d}_T$ is given by the content $d_1$ of its first vertex. Indeed, given $\mathfrak{d}_T$, we can infer that the color set $C_1$ of $M(T)$ must be one of the cyclic intervals $[2k-d_1+1,1],\dots,[1,d_1]$. Once such a $C_1$ is chosen, we can reconstruct the color sets $C_2,\dots,C_s$, and moreover the matching $M(T)$, from the data of $\mathfrak{d}_T$. 

We can say for brevity that the mapping is reversible ``up to rotations.''
\end{remark}

\subsection{From dissections to triangulations}
We now discuss the passage from weighted dissections of weight $k$ to weighted triangulations of weight $k$, i.e. the bottom dashed arrow in Figure~\ref{fig:schematic}. 

Given a weighted dissection $\mathfrak{d}$ of the $s$-gon, we say that a weighted triangulation $\mathfrak{t}$ {\sl extends} $\mathfrak{d}$ if the underlying triangulation for $\mathfrak{t}$ extends the underlying dissection for $\mathfrak{d}$ and if moreover, the weight function for $\mathfrak{t}$ agrees with that for $\mathfrak{d}$ on the diagonals they share and assigns zero weight to the diagonals not present in $\mathfrak{d}$. 

For a tableau $T \in {\rm SYT}(2\omega_k)$, we will use the notation $\mathfrak{t}_T$ to indicate {\sl any} choice of weighted triangulation which extends the weighted dissection $\mathfrak{d}_T$ defined in the previous section. Observe that if $\mathfrak{t}$ and $\mathfrak{t}'$ are two such extensions, then their underlying triangulations are related by a sequence of flip moves at diagonals not present in $\mathfrak{d}$.

\begin{example}
Continuing our running example, there are four choices of triangulations $\mathfrak{t}$ extending the dissection $\mathfrak{d}$ from Figure~\ref{fig:intermediateobjects} because we can flip either of the dashed diagonals. 
\end{example}

\subsection{From triangulations to web diagrams}
We now construct a map $\mathfrak{t} \mapsto W(\mathfrak{t})$
from weight~$k$ weighted triangulations of an $s$-gon to standard ${\rm SL}_k$ webs of degree 2. That is, we formulate the upwards arrow in Figure~\ref{fig:schematic}. To be more precise, because of the rotational ambiguity discussed in Remark~\ref{rmk:reversibledissection}, the input to our map is a pair $(\mathfrak{t}, C_1)$ where $C_1 \in \binom{[2k]}{d_1}$ is a set whose cardinality matches the content $d_1$ of the first vertex of $\mathfrak{t}$ and which satisfies $1 \in C_1$. 

\begin{remark}\label{rmk:intended} In our intended application, the weighted triangulation $\mathfrak{t}$ is an extension of a weighted dissection $\mathfrak{d}_T$ corresponding to a tableau~$T$, and the set $C_1$ is the corresponding color set of the matching $M(T)$. \end{remark}

{\bf Construction.} Let $(d_1,\dots,d_s)$ be the content of $\mathfrak{t}$ and $C_1$ be as just described. For $i=2,\dots,s$, define inductively a cyclic interval $C_i \subset [2k]$ which contains the next $d_i$ many vertices immediately following the interval $C_{i-1}$. Thus $C_1 \coprod \cdots \coprod C_s$ is a decomposition of $[2k]$ into cyclic intervals, using \eqref{eq:dissectiondegreesum}.

We first describe $W(\mathfrak{t})$ as a graph and then address its edge multiplicities. The graph $W(\mathfrak{t})$ has 2k many boundary vertices, has $s$ many interior white vertices $w_1,\dots,w_s$ and has interior trivalent black vertices $b(t)$ indexed by the triangles $t$ in $\mathfrak{t}$. If $t$ has 
has boundary vertices $a,b,c$ in~$\mathfrak{t}$, then the trivalent vertex $b(t)$ is joined to the white vertices $w_a,w_b,w_c$ by edges. Finally, each interior white vertex $w_i$ is incident to the $d_i$ many boundary vertices in the set $C_i$ by a single edge of multiplicity one. This completes the description of $W(\mathfrak{t})$ as a graph. 

The multiplicities of the edges $b(t)w_a$, $b(t)w_b$, and $b(t)w_c$ around a trivalent vertex are specified as follows. Note that $\mathfrak{t} \setminus {\rm int}(t)$ is a union of three weighted triangulations, one for each side of $t$. We let $A$ denote the triangulation which contains the side $w_bw_c$, i.e. the triangulation which is ``opposite'' the edge $b(t)w_a$.
Then we set 
$${\rm mult}_{W(\mathfrak{t})}(b(t)w_a)  = \sum_{e \in \text{ polygon $A$ of $\mathfrak{t}$}}{\rm wt}_T(e),$$
with the sum over all edges $e$ inside the polygon $A$. The multiplicities of the edges $b(t)w_b$ and $b(t)w_c$ are defined analogously using the triangles $B,C$ opposite the sides $b(t)w_b$ and $b(t)w_c$ respectively. This completes the definition of $W(\mathfrak{t})$.

\medskip

Let us check that this above construction indeed produces an ${\rm SL}_k$ web of degree two. Since each diagonal $e \in \mathfrak{t}$ appears in exactly one of the triangulations $A,B,C$, the multiplicity sum around each $b(t)$ equals~$k$ as required. We must also argue that the multiplicity sum around every white vertex equals~$k$. Indeed, summing the multiplicities of the edges from $w_i$ to interior black vertices counts the weight of all of the diagonals of $\mathfrak{t}$ except those which are incident to boundary vertex $i$, i.e. this sum equals $k-d_i$. But we have added exactly $d_i$ additional edges of multiplicity one from $w_i$ to boundary vertices so that the total multiplicity sum equals $k$ indeed.  
 
{\bf Summary.} The tableau-to-web map is the map 
\begin{equation}\label{eq:isitdfd}
T \mapsto W(\mathfrak{t}_T) 
\end{equation}
where $\mathfrak{t}_T$ is {\sl any} choice of triangulation extending the dissection $\mathfrak{d}_T$. Recall that the construction of the web $W(\mathfrak{t})$ from a weighted triangulation $\mathfrak{t}$ requires a choice of set $C_1$. As alluded to above, in \eqref{eq:isitdfd} we are implicitly choosing this $C_1$ to be the color set $C_1$ of the matching $M(T)$.

Our next lemma says that the web invariant $[W(\mathfrak{t}) ]$ is independent of the choice of triangulation $\mathfrak{t}_T$ extending the dissection $\mathfrak{d}_T$.

\begin{proposition}\label{prop:flipinvariance}
Let $\mathfrak{d}$ be a weighted dissection of weight $k$ and let $\mathfrak{t},\mathfrak{t}'$ be two extensions of $\mathfrak{d}$ to weighted triangulations. 
Then the corresponding ${\rm SL}_k$ web invariants coincide:
$$[ W(\mathfrak{t}) ] = [ W(\mathfrak{t}') ] \in L_{2k}(2\omega_k) \cong \mathbb{C}[{\rm Gr}(k,2k)]_{(2)}.$$
\end{proposition}

\begin{proof}
It suffices to show the claim when $\mathfrak{t}$ and $\mathfrak{t}'$
differ by a flip in a quadrilateral since triangulations which extend~$\mathfrak{d}$ are connected by such moves. The diagonal of such quadrilateral is not in~$\mathfrak{d}$ and therefore has zero weight in both $\mathfrak{t},\mathfrak{t}'$. The local picture around the quadrilateral is illustrated in Figure~\ref{fig:skein} where $\alpha$, $\beta$, $\gamma$, $\delta$ are the weights of the four regions outside the quadrilateral. See Figure~\ref{fig:examplesquaremove} for an application of this local picture in our running example.  

The equality $[W(\mathfrak{t})] = [W(\mathfrak{t}')]$ then is an application of the ``square switch'' skein relation [FLL (eq 6.2)] with parameters $r:= k$, $j:= \alpha + \beta$,$s:= \alpha$, $v := \gamma$, $\ell := \beta + \gamma$. Note that $j-\ell+v-s = 0$, so the only term on the right hand side of \cite[Equation (6.2)]{FLL} is the case $t=0$ with coefficient $\binom 0 0 =1$.
\end{proof}

\begin{figure}\centering
\begin{tikzpicture}[scale = .8]
\begin{scope}[xshift = -5cm]
\node at (-1.5,0) {$\bullet$};
\node at (0,1.5) {$\bullet$};
\node at (0,-1.5) {$\bullet$};
\node at (1.5,0) {$\bullet$};
\draw [](-1.5,0)--(0,1.5)--(1.5,0)--(0,-1.5)--(-1.5,0);
\draw [dashed] (0,-1.5)--(0,1.5);
\draw (-1.5,0) .. controls (-1.75,2.7) and (-.5,2.3) ..  (0,1.5);
\draw (-1.5,0) .. controls (-1.75,-2.7) and (-.5,-2.3) ..  (0,-1.5);
\draw (1.5,0) .. controls (1.75,2.7) and (.5,2.3) ..  (0,1.5);
\draw (1.5,0) .. controls (1.75,-2.7) and (.5,-2.3) ..  (0,-1.5);
\node at (-.9,1.2) {$\alpha$};
\node at (-.9,-1.2) {$\beta$};
\node at (.9,1.2) {$\delta$};
\node at (.9,-1.2) {$\gamma$};
\end{scope}

\node at (-2,0) {$\circ$};
\node at (0,2) {$\circ$};
\node at (0,-2) {$\circ$};
\node at (2,0) {$\circ$};
\node at (-.66,0) {$\bullet$};
\node at (.66,0) {$\bullet$};
\draw (-1.95,0)--(-.66,0);
\draw (.66,0)--(1.95,0);
\draw (-.05,-1.95)--(-.66,0)--(-.05,1.95);
\draw (.05,-1.95)--(.66,0)--(.05,1.95);
\node at (-.8,-1) {$\alpha$};
\node at (-.8,1) {$\beta$};
\node at (.8,1) {$\gamma$};
\node at (.8,-1) {$\delta$};
\node at (-1.3,.3) {$\gamma+\delta$};
\node at (1.3,.3) {$\alpha+\beta$};
\begin{scope}[xshift = 5cm]
\node at (-2,0) {$\circ$};
\node at (0,2) {$\circ$};
\node at (0,-2) {$\circ$};
\node at (2,0) {$\circ$};
\node at (0,-.66) {$\bullet$};
\node at (0,.66) {$\bullet$};
\draw (0,-1.95)--(0,-.66);
\draw (0,.66)--(0,1.95);
\draw (-1.95,-.05)--(0,-.66)--(1.95,-.05);
\draw (-1.95,.05)--(0,.66)--(1.95,.05);
\node at (-.8,-.8) {$\gamma$};
\node at (-.8,.8) {$\delta$};
\node at (.8,.8) {$\alpha$};
\node at (.8,-.8) {$\beta$};
\node at (.65,-1.5) {$\alpha+\delta$};
\node at (.65,1.5) {$\beta+\gamma$};
\end{scope}
\end{tikzpicture}
\caption{The left picture indicates schematically a diagonal of $\mathfrak{t} \setminus \mathfrak{d}$ (drawn dashed) which can be flipped inside its quadrilateral to obtain a new triangulation $\mathfrak{t}'$ . The numbers $\alpha,\dots,\delta$ indicate the weights of the four polygons outside this quadrilateral. The middle picture and right diagrams are the local picture for the corresponding web diagrams $W(\mathfrak{t})$ and $W(\mathfrak{t}')$ inside the quadrilateral. One has $[W(\mathfrak{t})]=[W(\mathfrak{t}')]$ by a skein relation. See Figure~\ref{fig:examplesquaremove} for an instance of this for the web in our running example with $\alpha=6$ and $\beta = \gamma = \delta=1$.}\label{fig:skein}
\end{figure}
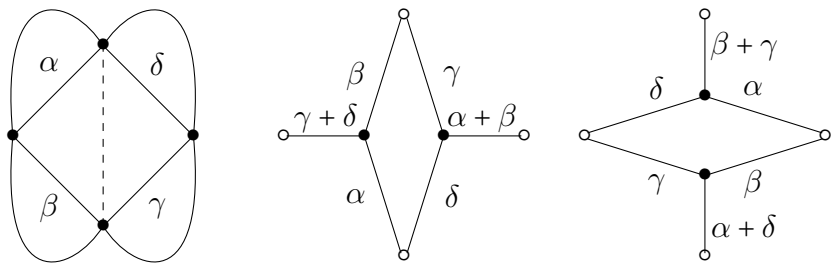

\begin{remark}
When the dissection $\mathfrak{d}_T$ happens to be a triangulation there is no freedom in our constructions. In particular, this happens whenever the number of short arcs of $M$ equals either two or three. Such $M$ are called {\sl tripartite pairings} and have recently arisen in a seemingly different mathematical context \cite{Jenne,JWY}. We think it should be possible to state the main identity from \cite{Jenne} as a quadratic relation between ${\rm SL}_2$ web immanants. 

When $M$ has exactly two short arcs the corresponding web invariant is a cluster monomial $\Delta_I\Delta_{[2k] \setminus I}$ for a cyclic interval $I \in \binom{[2k]}k$. Its web diagram is a disjoint union of two claw graphs. When $M$ has exactly three short arcs, the web diagram is a tree with one trivalent black vertex connected to three interior white vertices. This special case of our construction was previously formulated by I.~Le and E.~Yildirim who prove moreover that these web invariants are cluster variables. 
 \end{remark}
 
\begin{figure}[ht]\centering
\begin{tikzpicture}[scale = 1.1]
\begin{scope}[xshift=-7.5cm]
\draw (0:2.5cm)--(.4*360/18:1.85cm);
\draw (1*360/18:2.5cm)--(.6*360/18:1.85cm);
\node at (.5*360/18:1.8cm) {$\circ$};
\node at (0:2.5cm) {$\bullet$};
\node at (360/18:2.5cm) {$\bullet$};

\draw (16*360/18:2cm)--(16.4*360/18:1.6cm);
\draw (17*360/18:2cm)--(16.6*360/18:1.6cm);
\node at (16.5*360/18:1.5cm) {$\circ$};
\node at (16*360/18:2cm) {$\bullet$};
\node at (17*360/18:2cm) {$\bullet$};

\draw (2*360/18:2cm)--(-5+3*360/18:1.6cm);
\draw (4*360/18:2cm)--(5+3*360/18:1.6cm);
\draw (3*360/18:2cm)--(3*360/18:1.65cm);
\node at (3*360/18:1.5cm) {$\circ$};
\node at (2*360/18:2cm) {$\bullet$};
\node at (3*360/18:2cm) {$\bullet$};
\node at (4*360/18:2cm) {$\bullet$};

\draw (5*360/18:2cm)--(-5+6*360/18:1.6cm);
\draw (7*360/18:2cm)--(5+6*360/18:1.6cm);
\draw (6*360/18:2cm)--(6*360/18:1.65cm);
\node at (6*360/18:1.5cm) {$\circ$};
\node at (5*360/18:2cm) {$\bullet$};
\node at (6*360/18:2cm) {$\bullet$};
\node at (7*360/18:2cm) {$\bullet$};

\draw (8*360/18:2cm)--(-7.5+9.6*360/18:1.6cm);
\draw (9*360/18:2cm)--(-2.5+9.5*360/18:1.65cm);
\draw (10*360/18:2cm)--(2.5+9.5*360/18:1.65cm);
\draw (11*360/18:2cm)--(5+9.6*360/18:1.6cm);
\node at (9.5*360/18:1.5cm) {$\circ$};
\node at (8*360/18:2cm) {$\bullet$};
\node at (9*360/18:2cm) {$\bullet$};
\node at (10*360/18:2cm) {$\bullet$};
\node at (11*360/18:2cm) {$\bullet$};

\draw (12*360/18:2cm)--(-7.5+13.6*360/18:1.6cm);
\draw (13*360/18:2cm)--(-2.5+13.5*360/18:1.65cm);
\draw (14*360/18:2cm)--(2.5+13.5*360/18:1.65cm);
\draw (15*360/18:2cm)--(5+13.6*360/18:1.6cm);
\node at (13.5*360/18:1.5cm) {$\circ$};
\node at (12*360/18:2cm) {$\bullet$};
\node at (13*360/18:2cm) {$\bullet$};
\node at (14*360/18:2cm) {$\bullet$};
\node at (15*360/18:2cm) {$\bullet$};

\node at (-.8,.95) {\small 6};
\node at (-.2,1.1) {\small 2};
\node at (-.3,.1) {\small 2};
\node at (.3,-.28) {\textcolor{red}{\small 1}};
\node at (1.08,.6) {\textcolor{red}{\small 1}};
\node at (1.35,-.3) {\textcolor{red}{\small 1}};
\node at (-1.0,-.55) {\small 4};
\node at (-.9,.60) {\small 1};
\node at (-.4,-.95) {\small 3};
\node at (.35,-1.1) {\textcolor{red}{\small 2}};
\node at (.75,-.85) {\textcolor{red}{\small 6}};
\node at (1.45,.1) {\textcolor{red}{\small 7}};

\node at (6*360/18:.95cm) {$\bullet$};
\draw (6*360/18:1.4cm)--(6*360/18:.95cm);
\draw (-2+9.5*360/18:1.35cm)--(6*360/18:.95cm);
\draw (3*360/18:1.35cm)--(6*360/18:.95cm);

\node at (11.5*360/18:.6cm) {$\bullet$};
\draw (-2+9.5*360/18:1.35cm)--(11.5*360/18:.6cm);
\draw (3*360/18:1.35cm)--(11.5*360/18:.6cm);
\draw (13.5*360/18:1.35cm)--(11.5*360/18:.6cm);

\node[color=red] at (15*360/18:.75cm) {$\bullet$};
\draw[color=red] (-2+3*360/18:1.4cm)--(15*360/18:.75cm);
\draw[color=red] (13.5*360/18:1.35cm)--(15*360/18:.75cm);
\draw[color=red] (-2+16.5*360/18:1.35cm)--(15*360/18:.75cm);

\node at (.5*360/18:1.1cm) {\textcolor{red}{$\bullet$}};
\draw[color=red] (3*360/18:1.4cm)--(.5*360/18:1.1cm);
\draw[color=red] (.5*360/18:1.75cm)--(.5*360/18:1.1cm);
\draw[color=red] (16.5*360/18:1.4cm)--(.5*360/18:1.1cm);

\node at (360/18:2.75cm) {7};
\node at (2*360/18:2.4cm) {6};
\node at (3*360/18:2.4cm) {5};
\node at (4*360/18:2.4cm) {4};
\node at (5*360/18:2.4cm) {3};
\node at (6*360/18:2.4cm) {2};
\node at (7*360/18:2.4cm) {1};
\node at (8*360/18:2.4cm) {18};
\node at (9*360/18:2.4cm) {17};
\node at (10*360/18:2.4cm) {16};
\node at (11*360/18:2.4cm) {15};
\node at (12*360/18:2.4cm) {14};
\node at (13*360/18:2.4cm) {13};
\node at (14*360/18:2.4cm) {12};
\node at (15*360/18:2.4cm) {11};
\node at (16*360/18:2.4cm) {10};
\node at (17*360/18:2.4cm) {9};
\node at (0:2.75cm) {8};
\end{scope}

\draw (0:2.5cm)--(.4*360/18:1.85cm);
\draw (1*360/18:2.5cm)--(.6*360/18:1.85cm);
\node at (.5*360/18:1.8cm) {$\circ$};
\node at (0:2.5cm) {$\bullet$};
\node at (360/18:2.5cm) {$\bullet$};

\draw (16*360/18:2cm)--(16.4*360/18:1.6cm);
\draw (17*360/18:2cm)--(16.6*360/18:1.6cm);
\node at (16.5*360/18:1.5cm) {$\circ$};
\node at (16*360/18:2cm) {$\bullet$};
\node at (17*360/18:2cm) {$\bullet$};

\draw (2*360/18:2cm)--(-5+3*360/18:1.6cm);
\draw (4*360/18:2cm)--(5+3*360/18:1.6cm);
\draw (3*360/18:2cm)--(3*360/18:1.65cm);
\node at (3*360/18:1.5cm) {$\circ$};
\node at (2*360/18:2cm) {$\bullet$};
\node at (3*360/18:2cm) {$\bullet$};
\node at (4*360/18:2cm) {$\bullet$};

\draw (5*360/18:2cm)--(-5+6*360/18:1.6cm);
\draw (7*360/18:2cm)--(5+6*360/18:1.6cm);
\draw (6*360/18:2cm)--(6*360/18:1.65cm);
\node at (6*360/18:1.5cm) {$\circ$};
\node at (5*360/18:2cm) {$\bullet$};
\node at (6*360/18:2cm) {$\bullet$};
\node at (7*360/18:2cm) {$\bullet$};

\draw (8*360/18:2cm)--(-7.5+9.6*360/18:1.6cm);
\draw (9*360/18:2cm)--(-2.5+9.5*360/18:1.65cm);
\draw (10*360/18:2cm)--(2.5+9.5*360/18:1.65cm);
\draw (11*360/18:2cm)--(5+9.6*360/18:1.6cm);
\node at (9.5*360/18:1.5cm) {$\circ$};
\node at (8*360/18:2cm) {$\bullet$};
\node at (9*360/18:2cm) {$\bullet$};
\node at (10*360/18:2cm) {$\bullet$};
\node at (11*360/18:2cm) {$\bullet$};

\draw (12*360/18:2cm)--(-7.5+13.6*360/18:1.6cm);
\draw (13*360/18:2cm)--(-2.5+13.5*360/18:1.65cm);
\draw (14*360/18:2cm)--(2.5+13.5*360/18:1.65cm);
\draw (15*360/18:2cm)--(5+13.6*360/18:1.6cm);
\node at (13.5*360/18:1.5cm) {$\circ$};
\node at (12*360/18:2cm) {$\bullet$};
\node at (13*360/18:2cm) {$\bullet$};
\node at (14*360/18:2cm) {$\bullet$};
\node at (15*360/18:2cm) {$\bullet$};

\node at (-.8,.95) {\small 6};
\node at (-.2,1.1) {\small 2};
\node at (-.3,.1) {\small 2};
\node at (-1.0,-.55) {\small 4};
\node at (-.9,.60) {\small 1};
\node at (-.4,-.95) {\small 3};

\node at (6*360/18:.95cm) {$\bullet$};
\draw (6*360/18:1.4cm)--(6*360/18:.95cm);
\draw (-2+9.5*360/18:1.35cm)--(6*360/18:.95cm);
\draw (3*360/18:1.35cm)--(6*360/18:.95cm);

\node at (11.5*360/18:.6cm) {$\bullet$};
\draw (-2+9.5*360/18:1.35cm)--(11.5*360/18:.6cm);
\draw (3*360/18:1.35cm)--(11.5*360/18:.6cm);
\draw (13.5*360/18:1.35cm)--(11.5*360/18:.6cm);

\node at (360/18:2.75cm) {7};
\node at (2*360/18:2.4cm) {6};
\node at (3*360/18:2.4cm) {5};
\node at (4*360/18:2.4cm) {4};
\node at (5*360/18:2.4cm) {3};
\node at (6*360/18:2.4cm) {2};
\node at (7*360/18:2.4cm) {1};
\node at (8*360/18:2.4cm) {18};
\node at (9*360/18:2.4cm) {17};
\node at (10*360/18:2.4cm) {16};
\node at (11*360/18:2.4cm) {15};
\node at (12*360/18:2.4cm) {14};
\node at (13*360/18:2.4cm) {13};
\node at (14*360/18:2.4cm) {12};
\node at (15*360/18:2.4cm) {11};
\node at (16*360/18:2.4cm) {10};
\node at (17*360/18:2.4cm) {9};
\node at (0:2.75cm) {8};

\node at (.9,.75) {\textcolor{red}{\small 2}};
\node at (1.2,.5) {\textcolor{red}{\small 6}};
\node at (1.45,-.25) {\textcolor{red}{\small 1}};
\node at (.2,-.3) {\textcolor{red}{\small 1}};
\node at (.4,-1.3) {\textcolor{red}{\small 1}};
\node at (.95,-1.0) {\textcolor{red}{\small 7}};

\node[color=red] at (15.65*360/18:1.1cm) {$\bullet$};
\node[color=red] at (1.0*360/18:.6cm) {$\bullet$};
\draw[color=red] (15.65*360/18:1.1cm)--(16.35*360/18:1.4cm);
\draw[color=red] (15.65*360/18:1.1cm)--(13.75*360/18:1.35cm);
\draw[color=red] (15.65*360/18:1.1cm)--(.5*360/18:1.75cm);
\draw[color=red] (1.0*360/18:.6cm)--(2.9*360/18:1.4cm);
\draw[color=red] (1.0*360/18:.6cm)--(13.55*360/18:1.35cm);
\draw[color=red] (1.0*360/18:.6cm)--(.5*360/18:1.75cm);

\end{tikzpicture}
\caption{The left web is $W(\mathfrak{t})$ from our running example, cf.~Figures~\ref{fig:mainexample} and~\ref{fig:intermediateobjects}. The right web is $W(\mathfrak{t}')$ where $\mathfrak{t}'$ is the weighted triangulation obtained by the flip move at the right dashed edge in $\mathfrak{t}$. These two webs are related by a square switch skein relation at the red edges, thus define the same web invariant.}
\label{fig:examplesquaremove}
\end{figure}
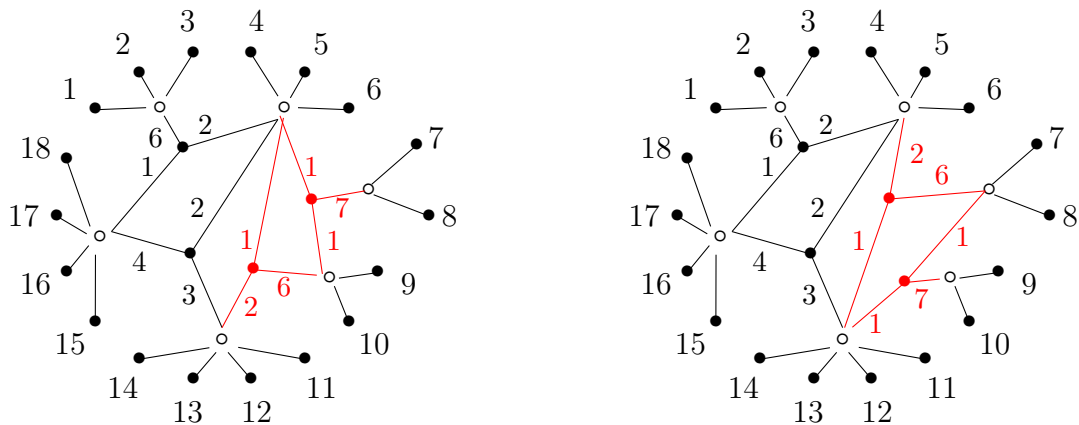
 
\section{Webs as plabic graphs}\label{secn:positroids}
Forgetting the data of the edge multiplicities, any type 2 web (and as a special case, any standard web) is in particular a plabic graph in the sense of \cite{Postnikov}. It is interesting to investigate how tight the connection between webs and plabic graphs is, see also \cite{FLL,HopkinsRubey}. In this section, we explore the tableau-to-web map as a tableau-to-plabic-graph map. The results in this section are mostly independent from the rest of the paper, although we will use the language of plabic graphs and dimer covers to a small extent in the proof of Theorem~\ref{thm:duality}.

We assume familiarity with the notions of positroids $\mathcal{M} \subset \binom{[n]}k$ and their encodings by decorated permutations $\pi$, Grassmann necklaces $\rmci$, or reduced plabic graphs~$G$ modulo move equivalence. We work always with graphs which are bipartite (not just bicolored) with boundary vertices colored black.

Let $G$ be a reduced plabic graph in the $n$-gon. A {\sl dimer cover} $\pi$ of $G$ is a 
subset $\pi \subset E(G)$ which covers each interior vertex once. We denote by $\partial(\pi) \subset [n]$ the subset of boundary vertices used in $\pi$. Observe that 
\begin{equation}\label{eq:picount}
\#\partial(\pi) = \# \text{ of white vertices of $G$} - \# \text{ of interior black vertices of $G$}.
\end{equation}
In particular this number depends only on $G$ not on $\pi$. We denote this number by $k$ and say that $G$ has type $(k,n)$. 
\begin{remark}
When $W = W(\mathfrak{t}_T)$ is a tableau-to-web image of some $T \in {\rm SYT}(2\omega_k)$, the number of white vertices of $W$ is the number $s = s(M(T))$ of short arcs while the number of interior black vertices is the number of triangles in a triangulation of the $s$-gon. It follows that the number~\eqref{eq:picount} equals~$2$ for $W$, i.e. that $W$ is a plabic graph of type~$(2,2k)$. 
\end{remark}

\begin{definition}[Cyclic interval positroids]\label{defn:cyclicpos}
Let $C_1,\dots,C_s$ be a decomposition of $[2k]$ into disjoint cyclic intervals. We obtain a positroid $\mathcal{M}_{C_1,\dots,C_s} \subset \binom{[2k]}2$ whose bases are exactly the pairs $\{i,j\}$ with the property that $i \in C_a$ and $j \in C_b$ with $a \neq b$.  
\end{definition}

The positroid subvariety of ${\rm Gr}(2,2k)$  corresponding to $\mathcal{M}_{C_1,\dots,C_s} \subset \binom{[2k]}2$ imposes the cyclic rank condition 
\begin{equation}\label{eq:rankconditions}
{\rm rank}(v_i \colon i \in C_a) \leq 1
\end{equation}
with one such rank condition for each $a=1,\dots,s$, where $v_1,\dots,v_{2k}$ are the column vectors of $2 \times 2k$ matrix representatives for points in ${\rm Gr}(2,2k)$.

\begin{proposition}\label{prop:whichpositroid}
Let $W(\mathfrak{t}_T)$ be an ${\rm SL}_k$ web diagram arising by evaluating the tableau-to-web map on $T \in {\rm SYT}(2\omega_k)$. Let $C_1,\dots,C_S$ be the color sets of the noncrossing matching $M(T)$. Then $W(\mathfrak{t}_T)$ is a reduced plabic graph whose corresponding positroid is $\mathcal{M}_{C_1,\dots,C_s}$.
\end{proposition}

\begin{proof}
Abbreviate $W = W(\mathfrak{t}_T)$. Let $C_j = (i_{j-1},i_{j}]$ be a color set of $M(T)$. Then the boundary vertices in $C_j$ are incident to the same white vertex $w_j$ of $W$. For $b = i_j,\dots,i_{j+1}-2$, the zizag path starting at $b$ (also known as the {\sl trip} starting at $b$) proceeds to $w_j$ and then ends at vertex $b+1$. It remains to understand the zigzag paths starting at $i_j-1$. The ``interior'' of~$W$ is the plabic graph dual to a triangulation of the $s$-gon and the trips starting at $i_j-1$'s are identified with trips in this smaller plabic graph. The zigzag paths of this smaller graph encode the top cell for ${\rm Gr}(2,s)$, i.e. each such zigzag path proceeds two units clockwise. Using this reasoning, one can see that the zigzag path of starting at $i_j-1$ ends at $i_{j+2}$. We now see that $G$ is reduced and moreover we have computed its trip permutation. Using the standard recipes to translate from trip permutations to rank conditions, we see that the rank conditions \eqref{eq:rankconditions} hold; on the other hand, by comparing with the top cell for ${\rm Gr}(2,s)$, we see that these are the {\sl only} rank conditions which hold, completing the proof. 
\end{proof}

We conclude this section with two remarks which explore further properties of the tableau-to-plabic-graph map.

\begin{remark}The data of $W$ as a plabic graph does not specify the corresponding canonical basis element $[W]$. For example, the pair of weighted triangulations
\begin{equation}\label{eq:samegraphdifferentweb}
\begin{tikzpicture}
\node at (1,0) {$\bullet$};
\node at (0,0) {$\bullet$};
\node at (0,1) {$\bullet$};
\node at (1,1) {$\bullet$};
\draw (0,0)--(1,0)--(1,1)--(0,1)--(0,0);
\draw (0,0)--(1,1);
\node at (.5,1.25) {\small 2};
\node at (.5,-.25) {\small 2};
\node at (-.25,.5) {\small 2};
\node at (1.25,.5) {\small 2};
\node at (.5,.75) {\small 1};
\node at (.5,-.75) {$\mathfrak{t}$};
\begin{scope}[xshift = 3cm]
\node at (1,0) {$\bullet$};
\node at (0,0) {$\bullet$};
\node at (0,1) {$\bullet$};
\node at (1,1) {$\bullet$};
\draw (0,0)--(1,0)--(1,1)--(0,1)--(0,0);
\draw (0,0)--(1,1);
\node at (.5,1.25) {\small 1};
\node at (.5,-.25) {\small 1};
\node at (-.25,.5) {\small 3};
\node at (1.25,.5) {\small 3};
\node at (.5,.75) {\small 1};
\node at (.5,-.75) {$\mathfrak{t}'$};
\end{scope}
\end{tikzpicture}
\end{equation}
determine ${\rm SL}_9$ webs $W(\mathfrak{t})$ and $W(\mathfrak{t}')$ with the same underlying plabic graph but with different edge multiplicities. As follows from Theorem~\ref{thm:duality}, these two triangulations are dual to distinct matchings of the $18$-gon, hence $[W(\mathfrak{t})$ and $[W(\mathfrak{t}')]$ correspond to distinct canonical basis elements. 
\end{remark}

\begin{remark}
Not every decomposition $[2k] = C_1 \coprod \cdots \coprod C_s$ into cyclic intervals arises from a noncrossing matching. For example, it is necessary that each $|C_i| \geq 2$. More substantively, there is no noncrossing matching of the octagon whose color sets have cardinalities $4$, $2$, and $2$. Even if we consider a set of cyclic intervals $C_1 \coprod \cdots \coprod C_s$ which {\sl do} come from a matching, not every reduced plabic graph for the positroid $\mathcal{M}_{C_1,\dots,C_s}$ can be assigned edge multiplicities in such a way that the corresponding web invariant is a canonical basis element. 

For instance, the plabic graph
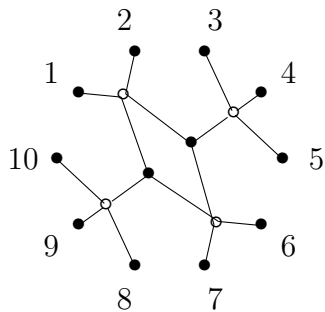
\begin{equation}\label{eq:notcanonical}
\begin{tikzpicture}
\node at (0*36:1.5cm) {$\bullet$};
\node at (0*36:1.95cm) {$5$};
\node at (1*36:1.5cm) {$\bullet$};
\node at (1*36:1.95cm) {$4$};
\node at (2*36:1.5cm) {$\bullet$};
\node at (2*36:1.95cm) {$3$};
\node at (3*36:1.5cm) {$\bullet$};
\node at (3*36:1.95cm) {$2$};
\node at (4*36:1.5cm) {$\bullet$};
\node at (4*36:1.95cm) {$1$};
\node at (5*36:1.5cm) {$\bullet$};
\node at (5*36:1.95cm) {$10$};
\node at (6*36:1.5cm) {$\bullet$};
\node at (6*36:1.95cm) {$9$};
\node at (7*36:1.5cm) {$\bullet$};
\node at (7*36:1.95cm) {$8$};
\node at (8*36:1.5cm) {$\bullet$};
\node at (8*36:1.95cm) {$7$};
\node at (9*36:1.5cm) {$\bullet$};
\node at (9*36:1.95cm) {$6$};
\draw (3*36:1.5cm)--(-2+3.5*36:1.05cm);
\draw (4*36:1.5cm)--(2+3.5*36:1.05cm);
\draw (9*36:1.5cm)--(2+8.5*36:1.05cm);
\draw (8*36:1.5cm)--(-2+8.5*36:1.05cm);
\draw (1*36:1.5cm)--(1*36:1.1cm);
\draw (0*36:1.5cm)--(-2+1*36:1.05cm);
\draw (2*36:1.5cm)--(2+1*36:1.05cm);
\draw (6*36:1.5cm)--(6*36:1.1cm);
\draw (5*36:1.5cm)--(-2+6*36:1.05cm);
\draw (7*36:1.5cm)--(2+6*36:1.05cm);
\node at (3.5*36:1.05cm) {$\circ$};
\node at (1*36:1.05cm) {$\circ$};
\node at (6*36:1.05cm) {$\circ$};
\node at (8.5*36:1.05cm) {$\circ$};
\node at (1*36:.35cm) {$\bullet$};
\node at (6*36:.35cm) {$\bullet$};

\draw (-2+3.5*36:1.05cm)--(1*36:.35cm);
\draw (2+3.5*36:1.05cm)--(6*36:.35cm);
\draw (-2+8.5*36:1.05cm)--(1*36:.35cm);
\draw (2+8.5*36:1.05cm)--(6*36:.35cm);
\draw (1*36:.95cm)--(1*36:.35cm);
\draw (6*36:.95cm)--(6*36:.35cm);

\end{tikzpicture}
\end{equation}
can only be made into an ${\rm SL}_5$ web diagram in a unique way up to reflection 
in the 10-gon. One can check that the resulting web diagram has nontrivial pairing with (at least) two different noncrossing matchings in the sense of Definition~\ref{defn:balancedwordcombinatorics} and is therefore not a canonical basis element. 

We can explain ``what goes wrong'' with this plabic graph as follows. There is a unique matching $M$ of the 10-gon whose color sets are $[1,2]$, $[3,4,5]$, $[6,7]$, and $[8,9,10]$. The corresponding dissection $\mathfrak{d}$ is a triangulated quadrilateral and the above plabic graph \eqref{eq:notcanonical} corresponds to the ``other'' triangulation of this quadrilateral, which is forbidden.
\end{remark}

\section{Web-matching duality}\label{secn:duality}
We start by reviewing some classical facts and some additional results from \cite{FLL}. 

We denote by $S(\lambda)$ the irreducible representation of the symmetric group $\mathfrak{S}_n$ indexed by a partition 
$\lambda$ of $n$. (Precisely, this means that the character $\chi_\lambda$ of $S(\lambda)$ satisfies the equality $\frac{1}{n!}\sum_{w \in \mathfrak{S}_n}\chi_\lambda(w)p_{\alpha(w)} = s_\lambda$ where $s_\lambda$ is the schur function, $\alpha(w)$ is the cycle type of $w$ written as a partition, and $p_{\alpha(w)}$ is the corresponding monomial in power sum symmetric functions.) Recall that each $\mathfrak{S}_n$-irrep is self-dual: $S(\lambda) \cong S(\lambda)^*$ where asterisk denotes duality of representations.

Let $\varepsilon = S(\omega_n)$ denote the sign representation which is 1-dimensional. If $\lam$ and $\mu$ are conjugate partitions one knows that $S(\lambda) \cong S(\mu) \otimes \varepsilon$. By Schur's lemma, there is therefore a unique $\mathfrak{S}_n$-equivariant map $S(\lambda)^* \to S(\mu) \otimes \varepsilon$, or equivalently, a unique equivariant pairing $S(\lambda) \otimes S(\mu) \to \varepsilon$.

The action of $\mathfrak{S}_n$ on $\mathbb{C}^n$ by coordinate permutation induces an action on the coordinate ring $\mathbb{C}[{\rm Gr}(k,n)]$. Let 
$\mathbb{C}[{\rm Gr}(k,dk)]^{\rm std}$ denote the subspace spanned by canonical basis vectors indexed by standard Young tableaux.
One knows that 
$$\mathbb{C}[{\rm Gr}(k,dk)]^{\rm std} \cong S(d\omega_k)$$
as $\mathfrak{S}_{dk}$-modules. From the previous paragraph, one has a canonical pairing 
\begin{equation}\label{eq:pairing}
\langle,\rangle \colon \, \mathbb{C}[{\rm Gr}(k,dk)]^{\rm std}\otimes \mathbb{C}[{\rm Gr}(d,dk)]^{\rm std} \to \varepsilon.
\end{equation}
We will use this pairing to show that our tableau-to-web map is 
``the right one.'' 

Our next definition allows for explicit computations with this pairing. 
\begin{definition}
Let $W$ be an ${\rm SL}_k$ web diagram. A {\sl consistent labeling} $\ell$ of~$W$ is a choice of subset $\ell(e) \in \binom{[k]}{{\rm mult}(e)}$ for each edge $e$ of $W$ subject to the requirement that $\ell(e)$ and $\ell(e')$ are disjoint whenever $e,e'$ are incident edges of $W$.
\end{definition}

Since the multiplicities around every vertex sum to $k$, the union of the subsets around every interior vertex equals $[k]$.

When $W$ is a standard ${\rm SL}_k$ web and $\ell$ is a consistent labeling, the labels of the boundary edges $e_i$ for $i=1,\dots,dk$ determine a sequence of singleton subsets $\ell(e_1),\dots,\ell(e_n)$ of $[k]$. One can show that each number from $1,\dots,k$ appears exactly $d$ many times in this sequence. We call this the {\sl boundary word} of $\ell$ and denote it by ${\bf w}(\ell)$.

We now introduce several combinatorial gadgets which one can attach to such a word. Consider the set 
$$\mathcal{A}_{k,d} = \{1_1,\dots,1_d,2_1,\dots,2_d,\cdots,k_1,\dots,k_d\}$$
with its symbols linearly ordered in the way we have just written them.

\begin{definition}\label{defn:balancedwordcombinatorics}
A $(k,d)$ {\sl balanced  word} {\bf w} is a word of length $dk$ in which each symbol from $1,\dots,k$ appears exactly $d$ times. 

We associate a sign to such word ${\bf w}$ as follows. We upgrade~{\bf w} to a permutation of the ordered set $\mathcal{A}_{k,d}$ by adding 
subscripts to the appearances of symbol $i$ so that subscripts increase from left to right in ${\bf w}$. Then we set {\sl sign}({\bf w}) to be the sign of this permutation. 

We associate a monomial in Pl\"ucker coordinates to such ${\bf w}$  as follows:
$$\Delta({\bf w}):= \prod_{j=1}^k \Delta_{\{i \in [dk] \colon {\bf w}_i = j\}} \in  \mathbb{C}[{\rm Gr}(d,dk)]^{\rm std}.$$

Finally, given a standard ${\rm SL}_k$ web $W$ of degree $d$, we define a natural number 
$$a(W,{\bf w}) = \#\{\ell \colon \, {\bf w}(\ell) = {\bf w}\},$$
the number of consistent labelings of $W$ with boundary word ${\bf w}$.
\end{definition}

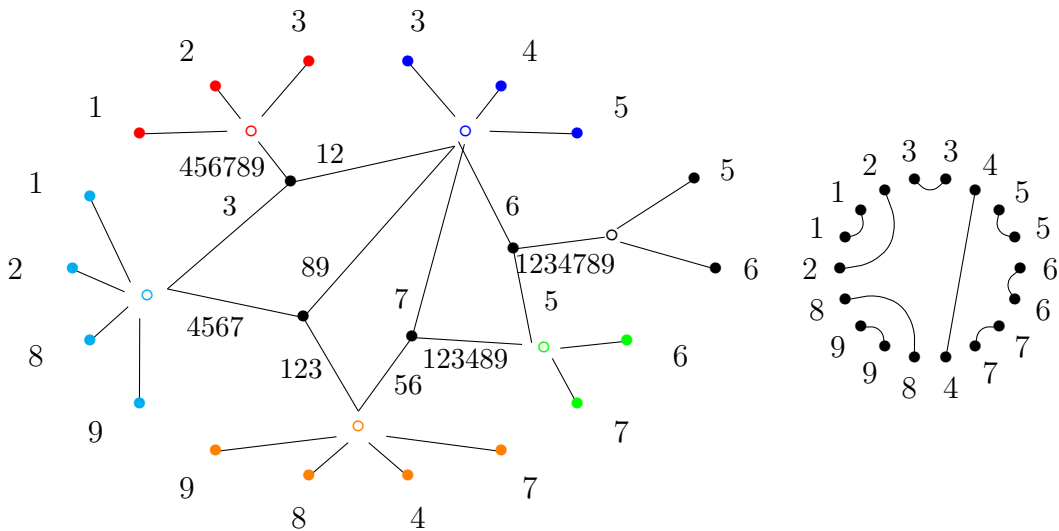
\begin{figure}\centering
\begin{tikzpicture}[xscale = 1.9,yscale = 1.4]
\draw (0:2.5cm)--(.4*360/18:1.85cm);
\draw (1*360/18:2.5cm)--(.6*360/18:1.85cm);
\node at (.5*360/18:1.8cm) {$\circ$};
\node at (0:2.5cm) {$\bullet$};
\node at (360/18:2.5cm) {$\bullet$};

\draw (16*360/18:2cm)--(16.4*360/18:1.6cm);
\draw (17*360/18:2cm)--(16.6*360/18:1.6cm);
\node at (16.5*360/18:1.5cm) {\textcolor{green}{$\circ$}};
\node at (16*360/18:2cm) {\textcolor{green}{$\bullet$}};
\node at (17*360/18:2cm) {\textcolor{green}{$\bullet$}};

\draw (2*360/18:2cm)--(-5+3*360/18:1.6cm);
\draw (4*360/18:2cm)--(5+3*360/18:1.6cm);
\draw (3*360/18:2cm)--(3*360/18:1.65cm);
\node at (3*360/18:1.5cm) {\textcolor{blue}{$\circ$}};
\node at (2*360/18:2cm) {\textcolor{blue}{$\bullet$}};
\node at (3*360/18:2cm) {\textcolor{blue}{$\bullet$}};
\node at (4*360/18:2cm) {\textcolor{blue}{$\bullet$}};

\draw (5*360/18:2cm)--(-5+6*360/18:1.6cm);
\draw (7*360/18:2cm)--(5+6*360/18:1.6cm);
\draw (6*360/18:2cm)--(6*360/18:1.65cm);
\node at (6*360/18:1.5cm) {\textcolor{red}{$\circ$}};
\node at (5*360/18:2cm) {\textcolor{red}{$\bullet$}};
\node at (6*360/18:2cm) {\textcolor{red}{$\bullet$}};
\node at (7*360/18:2cm) {\textcolor{red}{$\bullet$}};

\draw (8*360/18:2cm)--(-7.5+9.6*360/18:1.6cm);
\draw (9*360/18:2cm)--(-2.5+9.5*360/18:1.65cm);
\draw (10*360/18:2cm)--(2.5+9.5*360/18:1.65cm);
\draw (11*360/18:2cm)--(5+9.6*360/18:1.6cm);
\node at (9.5*360/18:1.5cm) {\textcolor{cyan}{$\circ$}};
\node at (8*360/18:2cm) {\textcolor{cyan}{$\bullet$}};
\node at (9*360/18:2cm) {\textcolor{cyan}{$\bullet$}};
\node at (10*360/18:2cm) {\textcolor{cyan}{$\bullet$}};
\node at (11*360/18:2cm) {\textcolor{cyan}{$\bullet$}};

\draw (12*360/18:2cm)--(-7.5+13.6*360/18:1.6cm);
\draw (13*360/18:2cm)--(-2.5+13.5*360/18:1.65cm);
\draw (14*360/18:2cm)--(2.5+13.5*360/18:1.65cm);
\draw (15*360/18:2cm)--(5+13.6*360/18:1.6cm);
\node at (13.5*360/18:1.5cm) {\textcolor{orange}{$\circ$}};
\node at (12*360/18:2cm) {\textcolor{orange}{$\bullet$}};
\node at (13*360/18:2cm) {\textcolor{orange}{$\bullet$}};
\node at (14*360/18:2cm) {\textcolor{orange}{$\bullet$}};
\node at (15*360/18:2cm) {\textcolor{orange}{$\bullet$}};

\node at (-.95,.97) {\small 456789};
\node at (-.2,1.1) {\small 12};
\node at (-.9,.60) {\small 3};
\node at (-1.0,-.55) {\small 4567};
\node at (1.45,.05) {\small 1234789};
\node at (1.35,-.3) {\small 5};
\node at (1.08,.6) {\small 6};
\node at (.75,-.85) {\small 123489};
\node at (.35,-1.1) {\small 56};
\node at (.3,-.28) {\small 7};
\node at (-.3,.02) {\small 89};
\node at (-.4,-.95) {\small 123};

\node at (6*360/18:.95cm) {$\bullet$};
\draw (6*360/18:1.4cm)--(6*360/18:.95cm);
\draw (-2+9.5*360/18:1.35cm)--(6*360/18:.95cm);
\draw (3*360/18:1.35cm)--(6*360/18:.95cm);

\node at (11.5*360/18:.6cm) {$\bullet$};
\draw (-2+9.5*360/18:1.35cm)--(11.5*360/18:.6cm);
\draw (3*360/18:1.35cm)--(11.5*360/18:.6cm);
\draw (13.5*360/18:1.35cm)--(11.5*360/18:.6cm);

\node at (15*360/18:.75cm) {$\bullet$};
\draw (-2+3*360/18:1.4cm)--(15*360/18:.75cm);
\draw (13.5*360/18:1.35cm)--(15*360/18:.75cm);
\draw (-2+16.5*360/18:1.35cm)--(15*360/18:.75cm);

\node at (.5*360/18:1.1cm) {$\bullet$};
\draw (3*360/18:1.4cm)--(.5*360/18:1.1cm);
\draw (.5*360/18:1.75cm)--(.5*360/18:1.1cm);
\draw (16.5*360/18:1.4cm)--(.5*360/18:1.1cm);


\node at (360/18:2.75cm) {5};
\node at (2*360/18:2.4cm) {5};
\node at (3*360/18:2.4cm) {4};
\node at (4*360/18:2.4cm) {3};
\node at (5*360/18:2.4cm) {3};
\node at (6*360/18:2.4cm) {2};
\node at (7*360/18:2.4cm) {1};
\node at (8*360/18:2.4cm) {1};
\node at (9*360/18:2.4cm) {2};
\node at (10*360/18:2.4cm) {8};
\node at (11*360/18:2.4cm) {9};
\node at (12*360/18:2.4cm) {9};
\node at (13*360/18:2.4cm) {8};
\node at (14*360/18:2.4cm) {4};
\node at (15*360/18:2.4cm) {7};
\node at (16*360/18:2.4cm) {7};
\node at (17*360/18:2.4cm) {6};
\node at (0:2.75cm) {6};

\node at (4,0) {
\begin{tikzpicture}[scale = .8]
\draw [rounded corners] (360/18:1.5cm)--(1.5*360/18:1.25cm)--(2*360/18:1.5cm);
\draw [rounded corners] (4*360/18:1.5cm)--(4.5*360/18:1.25cm)--(5*360/18:1.5cm);
\draw [rounded corners] (7*360/18:1.5cm)--(7.5*360/18:1.25cm)--(8*360/18:1.5cm);
\draw [rounded corners] (11*360/18:1.5cm)--(11.5*360/18:1.25cm)--(12*360/18:1.5cm);
\draw [rounded corners] (15*360/18:1.5cm)--(15.5*360/18:1.25cm)--(16*360/18:1.5cm);
\draw [rounded corners] (17*360/18:1.5cm)--(17.5*360/18:1.25cm)--(18*360/18:1.5cm);
\draw [rounded corners] (6*360/18:1.5cm)--(6.25*360/18:1.0cm)--(7.5*360/18:.75cm)--(8.75*360/18:1.0cm)--(9*360/18:1.5cm);
\draw [rounded corners] (10*360/18:1.5cm)--(10.25*360/18:1.0cm)--(11.5*360/18:.75cm)--(12.75*360/18:1.0cm)--(13*360/18:1.5cm);
\draw (3*360/18:1.5cm)--(14*360/18:1.5cm);
\node at (0:1.5cm) {$\bullet$};
\node at (360/18:1.5cm) {$\bullet$};
\node at (2*360/18:1.5cm) {$\bullet$};
\node at (3*360/18:1.5cm) {$\bullet$};
\node at (4*360/18:1.5cm) {$\bullet$};
\node at (5*360/18:1.5cm) {$\bullet$};
\node at (6*360/18:1.5cm) {$\bullet$};
\node at (7*360/18:1.5cm) {$\bullet$};
\node at (8*360/18:1.5cm) {$\bullet$};
\node at (9*360/18:1.5cm) {$\bullet$};
\node at (10*360/18:1.5cm) {$\bullet$};
\node at (11*360/18:1.5cm) {$\bullet$};
\node at (12*360/18:1.5cm) {$\bullet$};
\node at (13*360/18:1.5cm) {$\bullet$};
\node at (14*360/18:1.5cm) {$\bullet$};
\node at (15*360/18:1.5cm) {$\bullet$};
\node at (16*360/18:1.5cm) {$\bullet$};
\node at (17*360/18:1.5cm) {$\bullet$};

\node at (360/18:2cm) {5};
\node at (2*360/18:2cm) {5};
\node at (3*360/18:2cm) {4};
\node at (4*360/18:2cm) {3};
\node at (5*360/18:2.0cm) {3};
\node at (6*360/18:2.0cm) {2};
\node at (7*360/18:2.0cm) {1};
\node at (8*360/18:2cm) {1};
\node at (9*360/18:2cm) {2};
\node at (10*360/18:2cm) {8};
\node at (11*360/18:2cm) {9};
\node at (12*360/18:2cm) {9};
\node at (13*360/18:2cm) {8};
\node at (14*360/18:2cm) {4};
\node at (15*360/18:2cm) {7};
\node at (16*360/18:2cm) {7};
\node at (17*360/18:2cm) {6};
\node at (0:2cm) {6};
\end{tikzpicture}};
\end{tikzpicture}
\caption{A consistent labeling $\ell$ of the ${\rm SL}_9$ web $W$ from Figure~\ref{fig:mainexample}. We write e.g. $1234789$ to indicate $\ell(e) = \{1,2,3,4,7,8,9\}$. One sees each number $1,\dots,9$ around every interior vertex and the cardinality of the set $\ell(e)$ matches ${\rm mult}_W(e)$. Each number $i = 1,\dots,9$ appears exactly twice as on a boundary edge of this labeling; joining these two copies by an arc yields the noncrossing matching $M$ from Figure~\ref{fig:intermediateobjects} (redrawn at right for convenience). Web-matching duality asserts that $\ell$ is the unique labeling of $W$ whose boundary word is a noncrossing matching.} 
\label{fig:evalexample}
\end{figure}

\begin{example}\label{eg:labeling}
Figure~\ref{fig:evalexample} gives a consistent labeling $\ell$ of the ${\rm SL}_9$ web $W$ from Figure~\ref{fig:mainexample}. The boundary word of this labeling is 
\begin{equation}\label{eq:word}
{\bf w}(\ell) = 123345566774899821 \text{ of type (9,2)}.
\end{equation}
The corresponding monomial 
$$\Delta({\bf w}(\ell)) = \Delta_{1,18}\Delta_{2,17}\Delta_{3,4}\Delta_{5,12}\Delta_{6,7} \Delta_{8,9}\Delta_{10,11} \Delta_{13,16}\Delta_{14,15} $$ 
is a noncrossing monomial for ${\rm Gr}(2,n)$. 
\end{example}

\medskip

By \cite[Equation (5.16) and Theorem 8.1]{FLL}, the unique pairing $\langle, \rangle$ \eqref{eq:pairing} can be understood in terms of the numbers $a(W,{\bf w})$: 
\begin{lemma}
For a standard ${\rm SL}_k$ web $W$ of degree~$d$ and a $(k,d)$ balanced word ${\bf w}$, one has 
$$\langle [W], \Delta({\bf w})\rangle = a(W,{\bf w}).$$
\end{lemma}

We now specialize to the case $d=2$. A noncrossing matching~$M$ of the $2k$-gon gives us a $(k,2)$ balanced word ${\bf w}(M)$ with the property that the two locations of the symbol $i$ in ${\bf w}(M)$ are exactly the two endpoints of the $i$th arc of $M$ in the $2k$-gon. The word \eqref{eq:word} is this word when $M$ is the matching from Figures~\ref{fig:intermediateobjects} and \ref{fig:evalexample}. 
We abbreviate 
$$\Delta(M) := \Delta({\bf w}(M)).$$ 
Such monomials are called {\sl noncrossing monomials}. These monomials are the canonical basis for $\mathbb{C}[{\rm Gr}(2,2k)]^{\rm std}$.

We arrive at our main combinatorial assertion. 
\begin{theorem}[Web-matching duality]\label{thm:duality}
For $T \in {\rm SYT}(2\omega_k)$, let $W = W(\mathfrak{t}_T)$ be a web diagram constructed via our tableau-to-web map. Let $M = T(M)$ be the noncrossing matching corresponding to $T$ under the Catalan bijection. 

Then $a(W,{\bf w}(M)) = 1$ and $a(W,{\bf w}(M')) = 0$ for noncrossing matchings $M' \neq M$.
\end{theorem}

That is, the web invariant  $[W(\mathfrak{t}_T)]$ is dual to the noncrossing monomial $\Delta(M)$ under the pairing \eqref{eq:pairing}. Since this is true for any triangulation $ \mathfrak{t}_T$ extending the dissection $\mathfrak{d}_T$, we get an alternative proof that the tableau-to-web-invariant map is well-defined.

\begin{proof}
By induction on~$k$ with the base cases $k=1,2$ easily checked by hand. Suppose that $\ell$ is a consistent labeling of $W$ whose boundary 
word ${\bf w}(\ell)$ is the word ${\bf w}(M')$ of a noncrossing matching~$M'$ of the $2k$-gon. We will argue that $M' = M$. (The uniqueness of the labeling $\ell$ will also follow.)

Certainly, $M'$ has at least one short arc $i,i+1$. Suppose this is the $j$th arc of $M'$ when arcs are ordered by their left endpoints. This means that $\ell(e_i) = \ell(e_{i+1}) = j$. Thus, boundary vertices $i,i+1$ are incident to different white vertices of $W$, thus $i$ and $i+1$ reside in different color sets of $W$, i.e. 
$i,i+1$ is also a short arc of $M$. 

The edges $e \colon j \in \ell(e)$ form a dimer cover $\pi$ of $W$ whose boundary subset is given by $\partial (\pi) = \{i,i+1\}$. On the other hand, $\{i,i+1\}$ is a term in the Grassmann necklace $\rmci$ for the positroid $\mathcal{M}_{C_1,\dots,C_s}$ determined by the color sets $C_1,\dots,C_s$ of the matching~$M$. Thus $\pi$ is the unique such dimer cover of $W$ (see e.g. \cite[Proposition 5.13]{MSTwist}).

Given a dimer cover $\pi$ of a web $W$, we obtain an ${\rm SL}_{k-1}$ web $W \setminus \pi$ from $W$ by decrementing the multiplicity of each edge used in $\pi$. That is, $W \setminus \pi$ is the same underlying graph but has multiplicity function $${\rm mult}_{W \setminus \pi}(e) = \begin{cases}
& {\rm mult}_{W}(e)-1 \text{ if $e \in \pi$} \\
& {\rm mult}_{W}(e) \text{ if $e \not\in \pi$.}
\end{cases} $$
If ${\rm mult}_W(e) = 1$ and $e \in \pi$, one can interpret this operation as deleting the edge~$e$ from~$W$. In the case of the dimer cover $\pi$ from the previous paragraph, this means that the boundary edges incident to vertices $i,i+1$ are deleted, and $W \setminus \pi$ is naturally a degree two standard web on the ground set $[2k] \setminus \{i,i+1\}$. 

It follows that  
\begin{equation}\label{eq:first}
a(W,M') = a(W \setminus \pi,M' \setminus \{i,i+1\})
\end{equation}
where the notation $M' \setminus \{i,i+1\}$ stands for the noncrossing matching of the $2(k-1)$-gon on the ground set $[2k] \setminus \{i,i+1\}$ obtained by removing the arc $i,i+1$ from $M'$. In particular, $a(W \setminus \pi,M' \setminus \{i,i+1\}) \neq 0$.

To complete the proof, it would suffice to show that the ${\rm SL}_{k-1}$ web $W \setminus \pi$ is an image of the tableau-to-web map for the matching $M \setminus \{i,i+1\}$, since by induction we would be able to conclude that $M' \setminus \{i,i+1\} = M \setminus \{i,i+1\}$. As we will see, in the typical case, $W \setminus \pi$ is exactly the image of the tableau corresponding to the matching $M \ {i, i+1}$ under the tableau-to-web map. This fails in certain degenerate cases, but we will show that in these cases, $W \setminus \pi$ is equivalent to such an image.

Let $\mathfrak{d}$, $\mathfrak{t}$ be the dissection and chosen triangulation corresponding to $T$ and $W$. We analyze what happens when passing from $M$ to 
$M \setminus \{i,i+1\}$. 

The numbers $i,i+1$ are in adjacent color sets of $M$, hence determine a side $S$ of the dissection $\mathfrak{d}$ and the triangulation $\mathfrak{t}$. If the weight of this side is at least two in $\mathfrak{t}$, 
then we decrement the weight of $S$ by one in when passing from $M$ to $M \setminus \{i,i+1\}$. The two white vertices of $W$ corresponding to the side $S$ have one smaller degree, hence each is joined to one fewer boundary vertex (indeed the boundary edges incident to $i,i+1$ are deleted). We can choose the same triangulation $\mathfrak{t}$ (decrementing the weight of the corresponding side by one). Thus the underlying plabic graph does not change. We decrement the multiplicity of exactly one edge around each trivalent vertex (namely, we decrement the edge whose corresponding region contains the side~$S$). Altogether, in passing from $M$ to $M \setminus M \setminus \{i,i+1\}$ we delete the edges of a dimer cover of $W$ with boundary $i,i+1$. By the uniqueness, we must have deleted the edges in $\pi$, so that $W \setminus \pi$ indeed is an instance of the tableau-to-web map in this case. 

To complete the proof, we need to analyze what happens in passing from 
$M$ to $M \setminus \{i,i+1\}$ in the special case that the side $S$ has weight one in $\mathfrak{d}$. The color sets of $M \setminus \{i,i+1\}$ are obtained by deleting $i,i+1$ from the ground set and merging the two color sets containing these numbers. At the level of weighted dissections, we contract the side $S$ to a point and replace parallel weighted edges by a single edge while adding their weights. We can perform the same contraction operation to the weighted triangulation $\mathfrak{t}$ extending $\mathfrak{d}$. We delete exactly one triangle during this process, namely, the triangle containing the side $S$. At the level of webs, we delete a trivalent black vertex and merge the two white vertices which correspond to the deleted side $S$ while deleting their boundary legs at $i,i+1$. Let $W'$ be the web which results from these steps.

The webs $W\setminus \pi$ and $W'$ are related to each other by the following simple operation. Let $t$ be the triangle containing the side~$S$ in $\mathfrak{t}$. The corresponding trivalent black vertex $b(t) \in W$ has three edges, one of which is dual to the side $S$ and has multiplicity one. Any dimer cover of $W$ with boundary $i,i+1$ must use this edge. Thus, we delete this edge in the web $W \setminus \pi$ so that $b(t)$ becomes a bivalent vertex. It is not hard to see that $W'$ is obtained from $W \setminus \pi$ by contracting this bivalent vertex. Contracting bivalent vertices does not effect the number of consistent labelings so that $a(W',M' \setminus \{i,i+1\}) = a(W \setminus \pi ,M' \setminus \{i,i+1\}) \neq 0$. 
By the induction hypothesis applied to $W'$ we conclude that $M'\setminus \{i,i+1\} = M\setminus \{i,i+1\}$ completing the proof. 
\end{proof}

\section{Web invariants }\label{secn:invariants}
We briefly review the definition of web invariant here. 

We identify the coordinate ring $\mathbb{C}[{\rm Gr}(k,n)]$ with the algebra of ${\rm SL}_k$-invariant polynomial functions of $n$-tuples of vectors in $\mathbb{C}^k$. For example, the Pl\"ucker coordinate $\Delta_I$ is the ${\rm SL}_k$-invariant polynomial 
$$(v_1,\dots,v_n) \mapsto {\rm det}(v_{i_1},\dots,v_{i_k})$$
where $i_1,\dots,i_k$ are the elements of $I$ written in ascending order and $(v_1,\dots,v_n)$ is an $n$-tuple of vectors in $\mathbb{C}^k$.

As a special case, we have the identification 
$$\mathbb{C}[{\rm Gr}(k,dk)]^{\rm std} =  {\rm Hom}_{{\rm SL}_k}(\otimes_{i=1}^{dk} \mathbb{C}^k,\mathbb{C}).$$

By multilinearity, an element of the above hom-space is determined by 
how it evaluates on tensor products 
$e_{i_1} \otimes \cdots \otimes e_{i_{dk}} \in \otimes^{dk}\mathbb{C}^k$ where each $e_{i_j}$ is a standard basis vector for $\mathbb{C}^k$. 


\begin{definition}[Web invariant]\label{defn:bracketofW}
Let $W$ be a standard ${\rm SL}_k$ web diagram of degree $d$. Then the web invariant $[W] \in \mathbb{C}[{\rm Gr}(k,dk)]_{(d)}$ is the polynomial function defined by its evaluation on tensor products of basis vectors as follows:
\begin{equation}\label{eq:brakcetofW}
[W](e_{i_1}\otimes \cdots \otimes e_{i_{dk}}) = {\rm sign}({\bf w})a(W,{\bf w}) 
\end{equation}
for all words $w = i_1,\dots,i_{dk}$ drawn from $[k]$.
\end{definition}

\begin{remark}
We have only defined ${\rm sign}({\bf w})$ when $w$ is a balanced word, but this is a necessary condition for $a(W,{\bf w})$ to be nonzero, so the above definition parses.\end{remark}

It follows from \cite[Lemma 5.4]{FLL} that the function defined in this way indeed determines an element of $\mathbb{C}[{\rm Gr}(k,dk)]^{\rm std}$. 

When $W$ is the claw graph on vertices $1,\dots,k$ (see Example~\ref{eg:clawgraph}), one has $[W] = \Delta_{1,\dots,k}$ and \eqref{eq:brakcetofW} encodes the Liebnitz formula for the determinant. 

\medskip

{\bf Warning:} Definition~\ref{defn:bracketofW} is a slick way of defining $[W]$ for all standard ${\rm SL}_k$ webs $W$, but it sometimes is the ``wrong'' definition up to a sign. For example, this definition can clash with  
the more familiar way of turning an  
${\rm SL}_2$ or ${\rm SL}_3$ web into a web invariant up to a multiplicative factor of $-1$. Thus, we refer to both $[W]$ and $-[W]$ as {\sl web invariants} when stating our two main theorems. Our next lemma identifies this sign discrepancy when $k=2$ and allows us to identify the ``correct'' sign for the webs $W(\mathfrak{t}_T)$ in Section~\ref{secn:sszn}.

A noncrossing matching $M$ of the $2k$-gon naturally becomes a standard ${\rm SL}_2$  web~$W_M$ by replacing each arc $ij \in M$ by a bivalent interior white vertex joined to vertices $i$ and $j$ by an edge of multiplicity one. The reader should not confuse the ${\rm SL}_2$ web $W_M$ with the ${\rm SL}_k$ webs $W(\mathfrak{t}_T)$ which we have introduced in Section~\ref{secn:construction}.

Our next result compares the web invariant $[W_m]$ to the noncrossing monomial $\Delta(M)$.

\begin{lemma}\label{lem:signs}
Let $T \in {\rm SYT}(2\omega_k)$ with entries $i_1 ,i_2 ,\dots,i_k$ in the first column. Let $M = M(T)$. 
Then 
$$[W_{M}] = \prod_{j\in [k]}(-1)^{i_j-j} \Delta(M) \in \mathbb{C}[{\rm Gr}(2,2k)].$$
\end{lemma}
\begin{proof}
We get a balanced word ${\bf w}(T)$ of type $(2,k)$ by putting 1's positions $i_1,\dots,i_k$ and putting $2$'s in positions $[2k] \setminus \{i_1,\dots,i_k\}$. Thus every left endpoint of $M$ is marked with a $1$ and every right endpoint is marked with a $2$. It follows that $\Delta(M)(e({\bf w}(T)) = +1$. 

On the other hand, $[W_{M(T)}](e({\bf w}(T)) = {\rm sign}({\bf w}(T))$. The result follows by checking that 
$${\rm sign}({\bf w}(T)) = \prod_{j\in [k]}(-1)^{i_j-j}$$
which is easy to see. 
\end{proof}


\section{Web immanants}\label{secn:immanants}
We briefly review ${\rm SL}_2$ web immanants introduced by T.~Lam \cite{LamDimers} in the spirit of \cite{RhoadesSkandera}.

A {\sl plabic network} is a choice of reduced plabic graph $G$ of type (k,n) together with an edge weighting ${\rm wt}_N \colon E(G) \to \mathbb{C}^*$. For a dimer cover $\pi$ of $G$ let ${\rm wt}_N(\pi):= \prod_{e \in \pi}{\rm wt}_N(e)$ be the product of the edge weights used in $\pi$. For $I \in \binom{[n]}k$, one has the complex number 
\begin{equation}\label{eq:DeltaI}
\Delta_I(N) := \sum_{\pi \colon \, \partial(\pi) = I}{\rm wt}_N(\pi),
\end{equation}
a weight-generating function for dimer covers $\pi$ of $N$ using boundary vertices~$I$.

For any plabic network $N$, the numbers $\Delta_I(N) \colon I \in \binom{[N]}k$ satisfy the Pl\"ucker relations, i.e. they determine a point in the Grassmannian 
${\rm Gr}(k,n)$. That is, for any $N$, one can find an $n$-tuple $(v_1,\dots,v_n)$ of vectors in $\mathbb{C}^k$ whose Pl\"ucker coordinates match the $\Delta_I(N)$'s:
\begin{equation}\label{eq:Ntovis}
{\rm det}(v_{i_1},\dots,v_{i_k}) = \Delta_I(N) \text{ for all $I \in \binom{[n]}k$}
\end{equation}
where again $i_1,\dots,i_k$ are the elements of $I$ in ascending order. It is known that the set of points in ${\rm Gr}(k,n)$ which arise from a network $N$ in this way is Zariski-dense.

The previous two paragraphs construct the canonical basis for $\mathbb{C}[{\rm Gr}(k,n)]_{(1)}$ using networks. The ${\rm SL}_2$ web immanants do the same for $\mathbb{C}[{\rm Gr}(k,n)]_{(2)}$.

The next definition is the ``semistandard analogue'' of a noncrossing matching.
\begin{definition}
A {\sl partial noncrossing matching} of type $(k,n)$ is a pair $(M,P)$ where $M$ is a noncrossing matching of some subset $S \subset [n]$ and $P \subset [n] \setminus S$ is a subset satisfying $\#S+2\#P = 2k$. The {\sl content} of a partial noncrossing matching $(M,P)$ of type $(k,n)$ is the sequence $(d_1,\dots,d_n)$ defined by $d_i = 1$ when $i \in S$, $d_i = 2$ when $i \in P$, and $d_i=0$ otherwise. \end{definition}
We think of elements of $P$ as the result of merging two vertices in a short arc to obtain a boundary vertex ``matched with itself.''

Given a plabic graph network $N$ of type $(k,n)$, a $2$-{\sl like subgraph} $W$ of $G$ is an ${\rm SL}_2$ web whose underlying plabic graph is a subgraph of $N$. Thus, $W$ is a disjoint union of edges of multiplicity two (henceforth often referred to as {\sl doubled edges}), cycles of even length formed by multiplicity-one edges, and multiplicity-one paths joining boundary vertices. Such $W$ naturally begets a type $(k,n)$ partial noncrossing matching $(M(W),P(W))$ where $M$ is the noncrossing matching of boundary vertices encoded by the paths of $W$ and $P$ is the subset of boundary vertices whose boundary edge is doubled in $W$. The content of $W$ matches the content of this partial noncrossing matching.

Note that because the definition of plabic graph requires  every boundary vertex to have degree at most one, hence any 2-like subgraph is a type 2 web in the sense of Section~\ref{secn:construction}.

Let $c(W)$ denote the number of cycles in~$W$. One defines 
\begin{equation}\label{eq:wt_sub_W_of_N}
{\rm wt}_N(W) = 2^{c(W)}\prod_{e \in E(N)}{\rm wt}_N(e)^{{\rm mult}_W(e)},
\end{equation}
a certain degree two analogue of the numbers ${\rm wt}_N(\pi)$.

By definition, the {\sl web immanant} $F_{(M,P)}$ indexed by a partial noncrossing matching $(M,P)$ of type $(k,n)$ is the element of $\mathbb{C}[{\rm Gr}(k,n)]_{(2)}$ which 
evaluates on vectors $(v_1,\dots,v_n)$ as follows. Choose a plabic network~$N$ of type $(k,n)$ whose Pl\"ucker coordinates match those of $(v_1,\dots,v_n)$ as in \eqref{eq:Ntovis}. Then 
$$F_{(M,P)}(v_1,\dots,v_n) = \sum_{M(W) = M, P(W)  = P}{\rm wt}_N(W),$$
where on the left hand side we evaluate the web immanant on $(v_1,\dots,v_n)$ and on the right hand side we sum over $2$-like subgraphs $W$ of~$N$ whose connectivity is given by $(M,P)$. By the aforementioned Zariski-denseness, this recipe indeed specifies $F_{(M,P)}$ as an element of the Grassmannian coordinate ring. 

By \cite[Theorems 2.1, 5.5, and 5.13]{LamDemazure}, the elements $F_{(M,P)}$
form a basis for $\mathbb{C}[{\rm Gr}(k,n)]_{(2)}$. This basis coincides with Lusztig's canonical basis for $L_n(2\omega_k)$ and inherits many good properties. The basis is cyclically invariant and has good restriction properties to positroid subvarieties of ${\rm Gr}(k,n)$. Moreover, each basis element is nonnegative when evaluated on points 
in ${\rm Gr}(k,n) \cap \mathbb{RP}^{\binom n k -1}_{\geq 0}$.

\section{Semistandardization}\label{secn:sszn}
The constructions in Sections~\ref{secn:construction}, \ref{secn:duality}, and \ref{secn:invariants} were formulated for standard tableaux. We now extend the tableau-to-web map to the setting of semistandard tableaux and prove Theorems~\ref{thm:main1} and~\ref{thm:main2}. 

The following encoding step appears often in this section. Let $(d_1,\dots,d_n)$ be a sequence drawn from $\{0,1,2\}$ and with $\sum_{i}d_i = 2k$. Then there is a unique weakly increasing sequence $a_1,\dots,a_{2k}$ in which 
each symbol $i=1,\dots,n$ appears exactly $d_i$ many times. We say that $(d_1,\dots,d_n)$ {\sl encodes to} the sequence $a_1,\dots,a_{2k}$.

\begin{definition}[Standardization of $T$]
The standardization of $T \in {\rm SSYT}(2\omega_k,[n])$ is the standard tableau $\hat{T} \in {\rm SYT}(2\omega_k)$ defined as follows. Let $(d_1,\dots,d_n)$ be the content of $T$ encoded as the sequence $a_1,\dots,a_{2k}$. One has a sequence of tableaux $\emptyset =:T_0  \subset \cdots \subset T_{2k}:= T$ via 
the following requirements: 
\begin{itemize}
\item each $T_i$ is obtained from its predecessor $T_{i-1}$ by adding a single box filled with the number $a_i$
\item if $a_i = a_{i+1}$, then the box $T_i \setminus T_{i-1}$ is in the first column.
\end{itemize}
The tableau $\hat{T}$ is defined by instead filling the box $T_i \setminus T_{i-1}$ with the number $i$ rather than the number~$a_i$.
\end{definition}

\begin{remark}
Recall that the first column of $T \in {\rm SYT}(2\omega_k)$ encodes left endpoints of arcs in the matching $M(T)$, i.e. the endpoint of the arc which is smaller in the order $1<2<\cdots<2k$. Standardization says that when $T$ has two copies of the same number, we should regard the entry that appears in the first column as ``smaller.'' 
\end{remark}

\begin{definition}[Semistandard tableau-to-web map]\label{defn::ssTtoW}
Let $T \in {\rm SSYT}(2\omega_k,[n])$ with content encoded by the sequence $a_1,\dots,a_{2k}$. Let 
$\hat{W}$ be a tableau-to-web image of the standardization $\hat{T}$. Thus $\hat{W}$ has boundary edges $e_1,\dots,e_{2k}$ and is drawn in the $2k$-gon. We construct a web $W$ in the $n$-gon which has the same interior vertices and edges as $\hat{W}$ and the same edge multiplicities, but whose edge $e_i$ is attached instead to the boundary vertex $a_i$ in the $n$-gon. By definition, such a web $W$ is a tableau-to-web image of $T$.
\end{definition}

Note that the content of the web diagram~$W$ constructed in this way matches the content of the semistandard tableau~$T$.

\begin{example}\label{eg:sszneg}
The semistandard tableau $T$ in Figure~\ref{fig:sszweb} 
standardizes to the standard tableau in Figure~\ref{fig:mainexample} that has served as our running example. It has content sequence
$$(d_1,\dots,d_{18})  =(1,1,2,1,1,0,1,1,2,1,1,1,1,0,2,1,1,1,0,1).$$ 
This sequence encodes to the sequence
\begin{equation}\label{eq:reencodes}
a_1,\dots,a_{18} = 1,2,3,3,4,5,7,8,8,9,10,11,12,14,14,15,16,18.
\end{equation}
To obtain a tableau-to-web image of~$T$, we reattach the 18 boundary legs of the web in Figure~\ref{fig:mainexample} according to this sequence as we have done in Figure~\ref{fig:sszweb}.
\end{example}

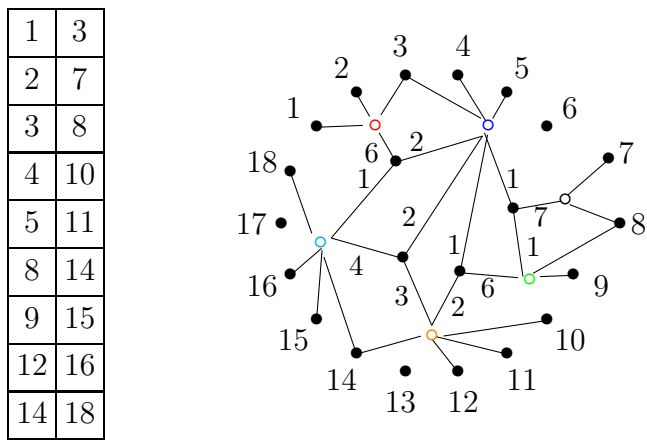
\begin{figure}\centering
\begin{tikzpicture}
\node at (-5,0) {\begin{ytableau}
1 & 3  \\
2 & 7  \\
3 & 8  \\
4 & 10  \\
5 & 11 \\
8 & 14  \\
9 & 15  \\
12 & 16  \\
14 & 18  \\
\end{ytableau}};
\draw (0:2.5cm)--(.4*360/18:1.85cm);
\draw (1*360/18:2.5cm)--(.6*360/18:1.85cm);
\node at (.5*360/18:1.8cm) {$\circ$};
\node at (0:2.5cm) {$\bullet$};
\node at (360/18:2.5cm) {$\bullet$};

\draw (17*360/18:2cm)--(4+16.5*360/18:1.6cm);
\draw (18*360/18:2.5cm)--(4+16.5*360/18:1.5cm);
\node at (16.5*360/18:1.5cm) {\textcolor{green}{$\circ$}};
\node at (16*360/18:2cm) {$\bullet$};
\node at (17*360/18:2cm) {$\bullet$};

\draw (3*360/18:2cm)--(3*360/18:1.62cm);
\draw (5*360/18:2cm)--(5+3*360/18:1.55cm);
\draw (4*360/18:2cm)--(2.5+3*360/18:1.55cm);
\node at (3*360/18:1.5cm) {\textcolor{blue}{$\circ$}};
\node at (2*360/18:2cm) {$\bullet$};
\node at (3*360/18:2cm) {$\bullet$};
\node at (4*360/18:2cm) {$\bullet$};

\draw (5*360/18:2cm)--(-5+6*360/18:1.6cm);
\draw (7*360/18:2cm)--(5+6*360/18:1.6cm);
\draw (6*360/18:2cm)--(6*360/18:1.65cm);
\node at (6*360/18:1.5cm) {\textcolor{red}{$\circ$}};
\node at (5*360/18:2cm) {$\bullet$};
\node at (6*360/18:2cm) {$\bullet$};
\node at (7*360/18:2cm) {$\bullet$};

\draw (8*360/18:2cm)--(-7.5+9.6*360/18:1.6cm);
\draw (10*360/18:2cm)--(2.5+9.5*360/18:1.5cm);
\draw (11*360/18:2cm)--(3.5+9.5*360/18:1.5cm);
\draw (12*360/18:2cm)--(5+9.6*360/18:1.5cm);
\node at (9.5*360/18:1.5cm) {\textcolor{cyan}{$\circ$}};
\node at (8*360/18:2cm) {$\bullet$};
\node at (9*360/18:2cm) {$\bullet$};
\node at (10*360/18:2cm) {$\bullet$};
\node at (11*360/18:2cm) {$\bullet$};

\draw (12*360/18:2cm)--(-7.5+13.6*360/18:1.5cm);
\draw (14*360/18:2cm)--(13.5*360/18:1.53cm);
\draw (15*360/18:2cm)--(2+13.5*360/18:1.5cm);
\draw (16*360/18:2cm)--(4+13.6*360/18:1.5cm);
\node at (13.5*360/18:1.5cm) {\textcolor{orange}{$\circ$}};
\node at (12*360/18:2cm) {$\bullet$};
\node at (13*360/18:2cm) {$\bullet$};
\node at (14*360/18:2cm) {$\bullet$};
\node at (15*360/18:2cm) {$\bullet$};

\node at (-.8,.95) {\small 6};
\node at (-.2,1.1) {\small 2};
\node at (-.3,.1) {\small 2};
\node at (.3,-.28) {\small 1};
\node at (1.08,.6) {\small 1};
\node at (1.35,-.3) {\small 1};
\node at (-1.0,-.55) {\small 4};
\node at (-.9,.60) {\small 1};
\node at (-.4,-.95) {\small 3};
\node at (.35,-1.1) {\small 2};
\node at (.75,-.85) {\small 6};
\node at (1.45,.1) {\small 7};

\node at (6*360/18:.95cm) {$\bullet$};
\draw (6*360/18:1.4cm)--(6*360/18:.95cm);
\draw (-2+9.5*360/18:1.35cm)--(6*360/18:.95cm);
\draw (3*360/18:1.35cm)--(6*360/18:.95cm);

\node at (11.5*360/18:.6cm) {$\bullet$};
\draw (-2+9.5*360/18:1.35cm)--(11.5*360/18:.6cm);
\draw (3*360/18:1.35cm)--(11.5*360/18:.6cm);
\draw (13.5*360/18:1.35cm)--(11.5*360/18:.6cm);

\node at (15*360/18:.75cm) {$\bullet$};
\draw (-2+3*360/18:1.4cm)--(15*360/18:.75cm);
\draw (13.5*360/18:1.35cm)--(15*360/18:.75cm);
\draw (-2+16.5*360/18:1.35cm)--(15*360/18:.75cm);

\node at (.5*360/18:1.1cm) {$\bullet$};
\draw (3*360/18:1.4cm)--(.5*360/18:1.1cm);
\draw (.5*360/18:1.75cm)--(.5*360/18:1.1cm);
\draw (16.5*360/18:1.4cm)--(.5*360/18:1.1cm);

\node at (360/18:2.75cm) {7};
\node at (2*360/18:2.4cm) {6};
\node at (3*360/18:2.4cm) {5};
\node at (4*360/18:2.4cm) {4};
\node at (5*360/18:2.4cm) {3};
\node at (6*360/18:2.4cm) {2};
\node at (7*360/18:2.4cm) {1};
\node at (8*360/18:2.4cm) {18};
\node at (9*360/18:2.4cm) {17};
\node at (10*360/18:2.4cm) {16};
\node at (11*360/18:2.4cm) {15};
\node at (12*360/18:2.4cm) {14};
\node at (13*360/18:2.4cm) {13};
\node at (14*360/18:2.4cm) {12};
\node at (15*360/18:2.4cm) {11};
\node at (16*360/18:2.4cm) {10};
\node at (17*360/18:2.4cm) {9};
\node at (0:2.75cm) {8};
\end{tikzpicture}
\caption{To obtain a tableau-to-web image of the semistandard tableau~$T$ at left, we standardize $T$ to obtain the tableau from Figure~\ref{fig:mainexample} and then reattach the legs of the web from Figure~\ref{fig:mainexample} according to the sequence \eqref{eq:reencodes}.
}\label{fig:sszweb}
\end{figure}
The invariant $[W]$ of a nonstandard web {\sl is defined} in the following way: 
\begin{equation}\label{eq:stitchW}
[W](v_1,\dots,v_n) = [\hat{W}](v_{a_1},\dots,v_{a_{2k}})
\end{equation}
where $a_1,\dots,a_{2k}$ is the sequence which encodes the content of $W$. Thus, we evaluate $[W]$ on $(v_1,\dots,v_n)$ by ``forgetting'' the vectors $v_i$ for which $d_i=0$ and using the vectors $v_i$ with $d_i=2$ ``twice.''

Next we turn to standardization of partial noncrossing matchings. Let $(M,P)$ be a partial noncrossing matching of type $(k,n)$ whose content is encoded by $a_1,\dots,a_{2k}$. We obtain a noncrossing matching $\widehat{(M,P)}$ uniquely by joining $i$ and $j$ in the $2k$-gon if either $a_i = a_j$ or $a_i \neq a_j$ and $a_i$ is joined to $a_j$ in $M$.

Note that the condition $a_i=a_j$ and $i<j$ only happens when in fact $j=i+1$. 

\begin{remark}\label{rmk:pcms}
To a pair $(M,P)$ as above, we associate a semistandard tableau $$T(M,P) \in {\rm SSYT}(2\omega_k,[n])$$ by first computing the tableau $\hat{T} := T(\widehat{(M,P)}) \in {\rm SYT}(2\omega_k)$ using the Catalan bijection, and then replacing the symbol $i \in \hat{T}$ with the symbol $a_i \in T$. It is simple to check that the map $(M,P) \mapsto T(M,P)$ is bijective. Thus, ${\rm SL}_2$ web immanants can be indexed by semistandard tableaux rather than by partial noncrossing matchings. 
\end{remark}

\begin{example}
The tableau $T$ from Example~\ref{eg:sszneg} corresponds to the partial noncrossing matching $(M,P)$ of type $(9,18)$ defined as follows. The set~$S$ from the definition of partial noncrossing matching is equal to $[18] \setminus \{3, 6,8,13, 14,17\}$.
The matching $M$ on the ground set $S$ consists of the pairs 
$\{1,18\}$, 
$\{2,16\}$, 
$\{4,11\}$, 
$\{5,7\}$, 
$\{9,10\}$, and 
$\{12,15\}$. The set~$P$ is $\{3,8,14\}$. Note for example that the content of $(M,P)$ agrees with the content of the tableau $T$.
\end{example}

\begin{lemma}\label{lem:stitchlemma}
For a partial noncrossing matching $(M,P)$ as above, one has 
\begin{equation}\label{eq:stitchFM}
F_{(M,P)}(v_1,\dots,v_n) = F_{\widehat{(M,P)}}(v_{a_1},\dots,v_{a_{2k}}).
\end{equation}
\end{lemma}

Here we have indexed the immanant on the left hand side by a partial noncrossing matching as in Remark~\ref{rmk:pcms}. 

The immanant on the left hand side of this equality is a function on ${\rm Gr}(k,n)$ (evaluated on the $n$-tuple of vectors $v_1,\dots,v_n \in \mathbb{C}^k$) while that on the right hand side is a function on ${\rm Gr}(k,2k)$. 

\begin{proof}
Given a $(k,n)$ plabic network~$N$,  and a content vector ${\bf d} = (d_1,\dots,d_n)$ drawn from $\{0,1,2\}$ and with $\sum d_i=2k$, we will define a network ${\rm split}_{\bf d}(N)$ of type $(k,2k)$ by the following steps. First, let $(a_1,\dots,a_{2k})$ be the weakly increasing sequence encoding ${\bf d}$. Second, recall that the definition of plabic graph mandates that each boundary vertex of $N$ has degree at most one. If boundary vertex $i$ has degree one, we denote by $v_i$ its neighboring interior vertex. Then ${\rm split}_{\bf d}(N)$ is obtained from $N$ by
\begin{itemize}
\item deleting all boundary edges $iv_i$ (if they are present),
\item adding boundary edges $iv_{a_i}$ for $i=1,\dots,2k$ whose weight matches the weight of the edge $iv_i$
\end{itemize}
That is, the $i$th boundary edge of $N$ gives rise to one or two boundary edges of ${\rm split}_{\bf d}(N)$ according to whether $d_i$ equals one or two. This recipe does not affect any interior vertices or edges. In particular, it does not change the quantity \eqref{eq:picount}, so that ${\rm split}_{\bf d}(N)$ indeed has type $(k,2k)$.

By a direct argument using matchings, it is straightforward to see that if $N$ has the same  Pl\"ucker coordinates as an $n$-tuple of vectors $(v_1,\dots,v_n)$, then ${\rm split}_{\bf d}$ has the same Pl\"ucker coordinates as the sequence $(v_{a_1},\dots,v_{2k})$. Thus, we can prove the desired equality \eqref{eq:stitchFM} by proving the equality of weight-generating functions for 2-weblike subgraphs of $N$ and ${\rm split}_{\bf d}(N)$.

 Next, let $W$ be a 2-like subgraph of $N$. Recall that $W$ is a type 2 web, and let ${\bf d} = (d_1,\dots,d_n)$ be its multiweight vector. Then ${\bf d}$ is a sequence as in the first paragraph. For any $i$ with $d_i=2$, there are consecutive indices $j,j+1$ such that $a_j=a_{j+1}=i$. We obtain a 2-like subgraph ${\rm split}_{\bf d}(W)$ of the network ${\rm split}_{\bf d}(N)$ by replacing each doubled boundary edge $iv_i$ of $W$ by a pair of multiplicity-one edges $jv_{i}$ and $(j+1)v_i$, setting the weights of these new edges equal to ${\rm wt}_W(iv_i)$. That is, we ``split up`` the doubled boundary edge in $W$ to two multiplicity-one edges in ${\rm split}_{\bf d}(W)$, joining the interior vertex $v_i$ to adjacent boundary vertices. (The rest of $W$ is left intact.) Clearly one has 
 $${\rm wt}_N(W) = {\rm wt}_{{\rm split}_{\bf d}(N)}({\rm split}_{\bf d}(W))$$ with ${\rm wt}$ as defined in \eqref{eq:wt_sub_W_of_N}.

Observe that whenever $a_j = a_{j+1} = i$ as in the previous paragraph, the boundary vertices $j,j+1$ are connected (in the matching-connectivity sense) by a multiplicity-one path in the 2-like subgraph ${\rm split}_{\bf d}(W)$. Observe also that any 2-like subgraph of ${\rm split}_{\bf d}(W)$ which connects these two boundary vertices must use the boundary edges $jv_i$ and $(j+1)v_i$.

Using the first observation in the previous paragraph, one can see that if $W$ has connectivity given by the partial noncrossing matching $(M,P)$, then ${\rm split}_{\bf d}(W)$ has connectivity given by the noncrossing matching $\widehat{(M,P)}$. Indeed, recall that the latter matching is obtained from $(M,P)$ by ``splitting'' each element of $P$ (a boundary vertex of the $n$-gon ``paired with itself'') into a short arc in the $2k$-gon. This is exactly what the passage $W \mapsto {\rm split}_{\bf d}(W)$ accomplishes at the level of connectivity. 

We have so far explained that every $2$-like subgraph $W$ of $N$ with connectivity $(M,P)$ begets a $2$-like subgraph ${\rm split}_{\bf}(W)$ of ${\rm split}_{\bf d}(N)$ with connectivity $\widehat{(M,P)}$ and with the same weight. Moreover, every 2-like subgraph of ${\rm split}_{\bf d}(N)$ with the correct connectivity arises in this way, using the second observation from two paragraphs previous. This establishes the desired equality of generating functions, completing the proof.
\end{proof}

We conclude by proving Theorems~\ref{thm:main1} and \ref{thm:main2}.
\begin{proof}
We will first show that every dual canonical basis element corresponding to a standard tableau is a web invariant and then deduce the statement for semistandard tableaux using Lemma~\ref{lem:stitchlemma}.

Consider a standard tableau $T \in {\rm SYT}(2\omega_k)$ and let $i_1,\dots,i_k$ be the entries in its first column. By Theorem~\ref{thm:duality}, the web invariant $[W(\mathfrak{t}_T)]$ is dual to the noncrossing monomial $\Delta(M) \in \mathbb{C}[\rm Gr(2,2k)]^{\rm std}$. On the other hand, by \cite[Equation (5.17)]{FLL}, the web immanant $F_{M}$ is dual to $[W_M]$ which equals $\prod_{j}(-1)^{i_j-1}\Delta(M)$ using Lemma~\ref{lem:signs}. (This is because ${\rm Web}_2(N)$ is a generating function for the $[W_M]$'s rather than for the $\Delta(M)$'s, see \cite[Definition 4.2 and equation (5.6)]{FLL}.) 

By comparing these, it follows that $\prod_{j}(-1)^{i_j-1}[W(\mathfrak{t}_T)]$ is a web immanant, i.e. a canonical basis element by \cite[Theorem 2.1 (2)]{LamDemazure}. This establishes Theorem~\ref{thm:main1} and equivalently \ref{thm:main2} for standard tableaux.

The web invariant for a semistandard tableau is defined in terms of those for standard tableaux by the evaluation equation \eqref{eq:stitchW}. We have checked that the same relationship holds between standard web immanants and semistandard ones in Lemma~\ref{lem:stitchlemma}. Thus each web invariant for semistandard tableaux is also a web immanant, which establishes \ref{thm:main2} and thus \ref{thm:main1}, completing the proof.

\end{proof}



\section*{Acknowledgements}

This project began during at the 2014 Cluster Algebras AMS MRC where we collaborated with Darlayne Addabbo, Eric Bucher, Sam Clearman, Laura Escobar, Ian Le, Ningning Ma, Suho Oh, and Hannah Vogel on a problem proposed by Thomas Lam and David Speyer. In particular, we had already conjectured Theorem~\ref{thm:main2} and recognized the importance of Definition~\ref{defn:intervals} at that time, see \cite[Appendix]{FLL}. We thank all of the above people for their contribution and thank especially Ian Le for many conversations which influenced our approach to this topic. We also thank the anonymous referee for helpful suggestions for improving the exposition and correctness.


\end{document}